\documentclass[12pt,reqno,draft]{article} 
\usepackage{amsmath,amssymb,amsthm,amsfonts, indentfirst, empheq}
\usepackage{enumerate,color,bm,graphicx,here}
\usepackage{multicol}

\makeatletter
\def\@cite#1#2{[{{\bfseries #1}\if@tempswa , #2\fi}]}
\renewcommand{\section}{%
\@startsection{section}{1}{\z@}
{0.5truecm plus -1ex minus -.2ex}%
{1.0ex plus .2ex}{\bfseries\large}}
\def\@seccntformat#1{\csname the#1\endcsname.\ }
\makeatother

\numberwithin{equation}{section} 
\newtheorem{thm}{Theorem}[section]

\newtheorem{lem}[thm]{Lemma}
\newtheorem{prop}[thm]{Proposition}
\theoremstyle{definition}
\newtheorem{df}[thm]{Definition}

\newtheorem{ex}[thm]{Example}

\newcommand{\ep}{\varepsilon}
\newcommand{\pa}{\partial}
\newcommand{\R}{\mathbb{R}}
\newcommand{\N}{\mathbb{N}}

\newcommand{\red}[1]{{\color{black}#1}}

\begin{document}
\footnote[0]
    {2020{\it Mathematics Subject Classification}\/. 
    Primary: 35K65; Secondary: 35K55, 34B15, 34C25.
    }
    
\footnote[0]
    {{\it Key words and phrases}\/: 
    chemotaxis; degenerate; volume-filling effect;
             stationary state.
    }
\begin{center}
    \Large{{\bf Global well-posedness and 
                     flat-hump-shaped stationary solutions
                     for degenerate chemotaxis systems 
                     with threshold density}}
\end{center}
\vspace{5pt}
\begin{center}
    Osuke Shibata,\quad 
    Tomomi Yokota%
    \footnote{Corresponding author.}%
    \footnote{Partially supported by 
                   JSPS KAKENHI Grant Number JP25K00917.}
   \footnote[0]{
    E-mail: 
    {\tt 1125522@ed.tus.ac.jp} (O.\,Shibata),\quad
    {\tt yokota@rs.tus.ac.jp} (T.\,Yokota)
    }\\
    \vspace{12pt}
    Department of Mathematics, 
    Tokyo University of Science\\
    1-3, Kagurazaka, Shinjuku-ku, 
    Tokyo 162-8601, Japan\\
    \vspace{2pt}
\end{center}
\begin{center}    
    \small \today
\end{center}

\vspace{2pt}
\newenvironment{summary}
{\vspace{.5\baselineskip}\begin{list}{}{%
     \setlength{\baselineskip}{0.85\baselineskip}
     \setlength{\topsep}{0pt}
     \setlength{\leftmargin}{12mm}
     \setlength{\rightmargin}{12mm}
     \setlength{\listparindent}{0mm}
     \setlength{\itemindent}{\listparindent}
     \setlength{\parsep}{0pt}
     \item\relax}}{\end{list}\vspace{.5\baselineskip}}
\begin{summary}
{\footnotesize {\bf Abstract.}
In a smoothly bounded domain $\Omega \subset \mathbb{R}^N$ 
$(N\in \mathbb{N})$,
a no-flux initial-boundary value problem for the degenerate chemotaxis system
with volume-filling effects,
\begin{align*}
    \begin{cases}
        u_t = \nabla \cdot (D(u,v) \nabla u - h(u,v) \nabla v),
        & \quad x\in \Omega, \  t>0,
        \\
        v_t = \Delta v + g(u,v),
        & \quad x\in \Omega, \  t>0,
    \end{cases}
\end{align*}
is considered  
under the assumptions that $D(1,s)=0$ and that $h(0,s)=h(1,s)=0$.
Here, initial data $u_0$ and $v_0$ have suitable regularity 
and satisfy 
$0\le u_0\le 1$ and $v_0\ge 0$ 
with $\nabla v_0 \cdot \nu|_{\partial \Omega} = 0$.
It is proved that there exists a global weak solution  
such that $0\le u\le 1$ and $v\ge 0$. 
Moreover, 
when $D(r,s) = D(r)$ 
for all $r\in[0,1]$ and $s\in[0,\infty)$
and additional conditions on $D$, $h$ and $g$ are assumed,
uniqueness of global weak solutions
with the mass conservation law
$\int_\Omega u(x,t) \, dx = \int_\Omega u_0(x) \, dx$
is shown.
Also,  
a flat-hump-shaped stationary solution
is constructed in the one-dimensional setting.}
\end{summary}
\vspace{10pt}

\newpage

\section{Introduction}\label{Sec:Intro}
\noindent
\textbf{Background.}
The property that cells move toward the location 
where the concentration of chemical substances is high is called 
\textit{chemotaxis}.
By chemotaxis,
the cell density near the chemical substances increases.
However, since volumes of cells are nonzero, 
the cell density has its maximal value 
and the cell movement stops at the value,
which is called \textit{volume-filling effects}
(cf.\ \cite{HP-2001-AAM}, \cite{PH-2002-CAMQ}).

\medskip
From a mathematical perspective,
there are some studies on existence and behavior of solutions to 
\textit{nondegenerate diffusive} chemotaxis systems incorporating volume-filling effects
(see e.g.\ \cite{W-2004-NA}, \cite{W-2006-PRSESA}, \cite{JZ-2009-AA},
\cite{W-2010-NA}, \cite{WWW-2011-N}).
On the other hand,
the \textit{degenerate diffusive} chemotaxis system with volume-filling effects, 
\begin{align*}
    \begin{cases}
        u_t = \nabla \cdot (D(u) \nabla u - h(u) \nabla v),
        \\
        v_t = \Delta v + g(u,v),
    \end{cases}
\end{align*}
under homogeneous Neumann boundary conditions and 
initial conditions
has been studied by Lauren\c{c}ot--Wrzosek \cite{LW-2005-PNDEA}.
Here, $u(x,t)$ and $v(x,t)$ represent cell density and 
chemical concentration, respectively.
Also, the diffusion coefficient $D$ and the sensitivity function $h$ satisfy
$D(1)=0$ and $h(0)=h(1)=0$.
In the literature, it has been shown that under some conditions on $D$, $h$ and $g$,
if initial data $u_0$ and $v_0$ fulfill 
$0\le u_0\le 1$ and $v_0\ge 0$, then there exists 
a global weak solution $(u,v)$ such that 
$0\le u\le 1$ and $v\ge 0$.
Moreover, under additional conditions on $D$, $h$ and $g$,
uniqueness of global weak solutions with the mass conservation law
$\int_\Omega u(x,t) \, dx = \int_\Omega u_0(x) \, dx$
has been established.
Recently, mathematical results on pattern formations 
in a related chemotaxis system have been provided 
in \cite{CIST-2020-DDSSB}.

\medskip
As for chemotaxis systems without volume-filling effects, 
the system
\begin{align*}
    \begin{cases}
        u_t =\nabla \cdot[(d_1+\chi v)\nabla u]
                -\chi \nabla \cdot (u\nabla v)+u(m_1-u+av),
        \\
        v_t =d_2\Delta v+v(m_2-bu-v)
    \end{cases}
\end{align*}
has been studied in \cite{HY-2023-NARWA},
where $d_1$, $d_2$, $m_1$, $\chi$, $a$ and $b$ are positive constants and 
$m_2$ is a real number. Also, 
the doubly degenerate chemotaxis system
\begin{align*}
    \begin{cases}
        u_t = \nabla \cdot (uv \nabla u) -\nabla \cdot (u^2v \nabla v),
        \\
        v_t = \Delta v+f(u)v
    \end{cases}
\end{align*}
has been considered in \cite{BKW-2026-JEE},
where $f$ satisfies the condition that 
$f(s)\ge c_f s^\alpha$ for $s\ge 1$ and 
$f(s)\le C_f s^\alpha$ for $s>0$,
with some $\alpha\in(0,\frac{2}{N})$.
Similar systems 
have already been investigated 
in \cite{W-2022-NARA}, \cite{W-2024-JDE} and \cite{CD-2025-AML}.
From these recent trends, 
it would be meaningful to analyze chemotaxis systems with diffusion and sensitivities depending not only on $u$ but also on $v$.
However, to the best of our knowledge, there is no study 
on chemotaxis systems with
volume-filling effects where diffusion coefficients 
and chemotactic sensitivity functions depend on both $u$ and $v$, 
whereas the cases independent of $v$ 
have recently been studied in
\cite{LZT-2025-AMO} and \cite{SCSS-2025-NARWA}.
These statements give rise to the natural question  
whether there are solutions 
of volume-filling chemotaxis models 
even when diffusion and sensitivities depend on both $u$ and $v$.

\medskip
\noindent
{\bf Main problem and results.}
We shall address chemotaxis systems 
with volume-filling effects 
when diffusion coefficients and chemotactic sensitivity functions depend on both $u$ and $v$.
Specifically, this paper focuses on the initial-boundary
value problem
\begin{align}\label{Sys:Main}
    \begin{cases}
        u_t = \nabla \cdot (D(u,v) \nabla u - h(u,v) \nabla v),
        & \quad x\in \Omega, \  t>0,
        \\
        v_t = \Delta v + g(u,v),
        & \quad x\in \Omega, \  t>0,
        \\
        (D(u,v) \nabla u - h(u,v) \nabla v) \cdot \nu = \nabla v \cdot \nu = 0,
        & \quad x\in \partial\Omega, \  t>0,
        \\
        u(\cdot, 0) = u_0, \ v(\cdot, 0) = v_0,
        & \quad x\in \Omega
    \end{cases}
\end{align}
in a smoothly bounded domain $\Omega \subset \R^N$ $(N \in \N)$,
where $\nu$ is the outward normal vector to $\pa\Omega$.
To make our overall hypotheses more precise 
we shall suppose that
\begin{align}
\label{Con:D}
   &\begin{cases}
          D \in C^2([0,1] \times [0,\infty)), 
          \quad D(1,s) = 0 \quad (s \in [0,\infty)), 
          \\
          \exists\, D_0 \in \red{C^1}([0,1]) \quad \mbox{s.t.} 
          \quad D(r,s) \ge D_0(r) > 0 \quad ((r,s) \in [0,1) \times [0,\infty)),
     \end{cases} 
   \\ 
\label{Con:h}
   &\begin{cases}
          h \in C^2([0,1] \times [0,\infty)), 
          \quad h(0,s) = h(1,s) = 0 \quad (s \in [0,\infty)), 
          \\
          h(r,s) > 0 \quad ((r,s) \in (0,1) \times [0,\infty)),
     \end{cases} 
   \\ 
\label{Con:g}
   &\begin{cases}
          g \in C^2([0,1] \times [0,\infty)), 
          \\
          g(r,0) \ge 0 \ \quad \mbox{and} 
          \quad \exists\, \kappa > 0 \quad \mbox{s.t.} 
          \quad g_s(r,s) \le \kappa 
          \quad \forall\, (r,s) \in [0,1] \times [0,\infty),
     \end{cases} 
   \\ 
\label{Con:ini}
   &\begin{cases}
          u_0 \in L^{\infty}(\Omega) 
          \quad \mbox{with} \quad 0 \le u_0 \le 1
          \ \mbox{a.e.\ in} \ \Omega, 
          \\
          v_0 \in L^{\infty}(\Omega) \cap H^{2}(\Omega)
          \quad \mbox{with} \quad \nabla v_0 \cdot \nu = 0 
          \ \mbox{on}\ \pa\Omega \quad \mbox{and} 
          \quad v_0 \ge 0 \ \mbox{a.e.\ in} \ \Omega.
    \end{cases}
\end{align}

Let us define a global weak solution of \eqref{Sys:Main}.
We let $C_{\mathrm{w}}([0,\infty);L^2(\Omega))$ 
denote the set of $L^2(\Omega)$-valued functions 
defined on $[0,\infty)$ which are continuous 
with respect to the weak topology in $L^2(\Omega)$.
\begin{df}[Global weak solutions] \label{Def:WS}
Let $u_0$ and $v_0$ satisfy \eqref{Con:ini}. 
Then a couple $(u,v)$ will be called a 
\textit{global weak solution} of \eqref{Sys:Main} if 
\begin{enumerate}[(a)]
   \item $u \in C_{\mathrm{w}}([0,\infty);L^2(\Omega)) \cap
                      L^\infty(\Omega \times (0,\infty))$, \ 
            $0 \le u \le 1$ a.e.\ in $\Omega \times (0,\infty)$,
   \item $v \in C([0,\infty);L^2(\Omega)) \cap
                      L^{2}_{\mathrm{loc}}([0,\infty);H^2(\Omega)) \cap
                 L^{\infty}_{\mathrm{loc}}(\overline{\Omega} \times [0,\infty))$, \ 
            $v \ge 0$ a.e.\ in $\Omega \times (0,\infty)$,
   \item $\mathcal{D}(u,v) \in L^{2}_{\mathrm{loc}}([0,\infty);H^1(\Omega))$, \ 
            $\mathcal{D}_s(u,v) \nabla v 
              \in L^{2}_{\mathrm{loc}}([0,\infty);(L^2(\Omega))^N)$,
            where
     \begin{equation} \label{def:hanaD}
            \mathcal{D}(r,s) := \int_0^r D(\sigma,s) \, d\sigma \quad \mbox{for}\ (r,s) \in \red{[0,1]} \times [0,\infty),
     \end{equation}
   \item for all $T > 0$ and $\varphi \in H^1(0,T;H^1(\Omega))$ 
            with $\varphi(T) = 0$,
     \begin{gather*}
            -\int_0^T \int_\Omega u \varphi_t \, dxdt 
            +\int_0^T \int_\Omega (D(u,v) \nabla u - h(u,v) \nabla v)
                                             \cdot \nabla \varphi \, dxdt
         = \int_\Omega u_0 \varphi(0) \, dx, 
         \\
            -\int_0^T \int_\Omega v \varphi_t \, dxdt 
            +\int_0^T \int_\Omega \nabla v \cdot \nabla \varphi \, dxdt 
            -\int_0^T \int_\Omega g(u,v)\varphi \, dxdt
         = \int_\Omega v_0 \varphi(0) \, dx, 
     \end{gather*}
            where $D(u,v)\nabla u
                        :=\nabla[\mathcal{D}(u,v)]-\mathcal{D}_s(u,v)\nabla v$.
\end{enumerate}
\end{df}
\noindent

Our main result reads as follows.
\begin{thm}[Existence] \label{Thm:exist}
Assume that \eqref{Con:D}, \eqref{Con:h}, \eqref{Con:g} 
and \eqref{Con:ini} hold.
\red{Then} the problem \eqref{Sys:Main} admits
a global weak solution $(u,v)$  
of \eqref{Sys:Main} such that
     \begin{equation} \label{shituryouhozonn}
          \int_\Omega u(x,t) \, dx = \int_\Omega u_0(x) \, dx 
          \quad \forall\, t \ge 0.
     \end{equation}
\end{thm}
\begin{thm}[Uniqueness] \label{Thm:unique}
Assume that $D(r,s) = D(r)$
for all $r\in[0,1]$ and $s\in[0,\infty)$
and that \eqref{Con:D}, \eqref{Con:h}, \eqref{Con:g} 
and \eqref{Con:ini} hold. 
Assume further that $v_0 \in W^{2,p}(\Omega)$ with $p>N$,
that
     \begin{equation} \label{SCon:Dh}
          \begin{cases}
               \forall\, K > 0 \ \exists\, C_0, C_1 > 0 \quad \mbox{s.t.} 
               \quad \forall\, (r_1, s_1), (r_2, s_2) \in [0,1] \times [0,K] 
               \\
               (h(r_1, s_1) - h(r_2, s_2))^2 
               \le C_0 (r_1 - r_2) (\mathcal{D}(r_1) - \mathcal{D}(r_2)) 
                     + C_1 (s_1 - s_2)^2, 
          \end{cases} 
     \end{equation}
where $\mathcal{D}(r) := \int_0^r D(\sigma) \, d\sigma$ for $r \in [0,1]$, 
and that
     \begin{equation} \label{SCon:g}
          \begin{cases}
               \exists\, C_2 > 0 \ \exists\, g_1, g_2 \in C^2 ([0,\infty)) 
               \ \mbox{with}  \ 
               g_1(0) \ge 0, \  g_2(0) \ge 0 \quad \mbox{s.t.}
               \\
               g(r,s) = g_1(s) + r g_2(s), 
               \quad \max \{g_1'(s), g_2'(s)\} \le C_2 
               \quad ((r,s) \in [0,1] \times [0,\infty)).
          \end{cases}
     \end{equation}
Then the global weak solution $(u,v)$ of \eqref{Sys:Main} 
satisfying \eqref{shituryouhozonn} is unique.
\end{thm}
Now, we provide examples fulfilling the assumption 
of Theorem \ref{Thm:exist} or Theorem \ref{Thm:unique}.
\begin{ex}
An example of $D$ and $h$ satisfying 
\eqref{Con:D} \red{and} \eqref{Con:h}
\red{in} Theorem \ref{Thm:exist} is given by
   $D(r,s) = (1 - r)^2 (s + 1)$ and 
   $h(r,s) = r (1 - r)^2 (s + 1)$.
Indeed, we have $\mathcal{D}(r,s) =- \frac13 (s+1) [(1-r)^3-1]$.
Therefore, for each $K > 0$, we have
$|\mathcal{D}_s(r,s)| \le \frac{K+1}{3}$
for every $r \in [0,1]$ and $s \in [0,K]$.
\end{ex}
\begin{ex}
An example of $D$ and $h$ satisfying 
\eqref{Con:D}, \eqref{Con:h} and \eqref{SCon:Dh} 
in Theorem \ref{Thm:unique} is given by
   $D(r,s) = (1 - r)^2$ and
   $h(r,s) = r (1 - r)^2 (s + 1)$.
In order to see that $D$ and $h$ fulfill \eqref{SCon:Dh}
we let $K > 0$ and fix $(r_1,s_1), (r_2,s_2) \in [0,1] \times [0,K]$. 
Then
\begin{align*}
   | h(r_1,s_1)-h(r_2,s_2) | 
   &\le | h(r_1,s_1)-h(r_2,s_1) | + | h(r_2,s_1)-h(r_2,s_2) | 
   \\
   &\le \left| \int_{r_2}^{r_1} |h_r(x,s_1)| \, dx \right|
          +\left| \int_{s_2}^{s_1} |h_s(r_2,y)| \, dy \right|.
\end{align*}
Since $h_r(r,s) = (1-r) (1-3r) (s+1)$ and 
$h_s(r,s)= r(1-r)^2$, 
we see that
\begin{align*}
   \left| \int_{r_2}^{r_1} |h_r(x,s_1)| \, dx \right|
   &\le 2 (K+1) \left| \int_{r_2}^{r_1} |1-x| \, dx \right| 
   \\
   &\le 2 (K+1) | r_1-r_2 |^{\frac12} 
          \left| \int_{r_2}^{r_1} |1-x|^2 \, dx \right|^{\frac12}
\end{align*}
and 
\[
    \left| \int_{s_2}^{s_1} |h_s(r_2,y)| \, dy \right|
    = \left| \int_{s_2}^{s_1} r_2 (1-r_2)^2 \, dy \right|
    \le | s_1-s_2 |.
\]
Therefore, we obtain
\[
    (h(r_1,s_1)-h(r_2,s_2))^2 
    \le 8 (K+1)^2 (r_1-r_2) (\mathcal{D}(r_1)-\mathcal{D}(r_2)) 
         +2 (s_1-s_2)^2.
\]
\end{ex}

\noindent
{\bf Key idea of the proof.}
Theorem \ref{Thm:exist} is proved by convergence of solutions
to approximate systems.
To this end, we need 
several uniform estimates for approximate solutions.
The key in the proof 
is how to deal with the partial derivative of $\mathcal{D}(u,v)$
with respect to the second variable.
If $D$ does not depend on $v$ as in \cite{LW-2005-PNDEA}, 
one can obtain the simple relation 
$\nabla[\mathcal{D}(u)]=D(u)\nabla u$.
However, since $D$ depends on both $u$ and $v$ in this paper, 
such simple relation breaks down and we have
$\nabla[\mathcal{D}(u,v)]=D(u,v)\nabla u+\mathcal{D}_s(u,v)\nabla v$.
Hence, to handle the additional term 
$\mathcal{D}_s(u,v)\nabla v$, we suggest the assumption
$|\mathcal{D}_s(r,s)|\le M$.

The idea in the proof of Theorem \ref{Thm:unique}
is to assume the new condition \eqref{SCon:Dh}, 
which is a generalization of \cite[Condition (9)]{LW-2005-PNDEA}.
Since the left-hand side of the condition \eqref{SCon:Dh} 
depends on the second variable,
we need terms depending on the second variable 
in addition to the first term
$C_0 (r_1 - r_2) (\mathcal{D}(r_1) - \mathcal{D}(r_2))$ 
on the right-hand side.
Also, since there are terms 
$\int_\Omega (v-\widehat{v})^2 \, dx$ and
$\int_0^t \int_\Omega (v-\widehat{v})^2 \, dxds$ in the proof 
of Theorem \ref{Thm:unique},
where $v$ and $\widehat{v}$ are solutions of \eqref{Sys:Main},
we can apply the Gronwall lemma by adding the term $C_1(s_1-s_2)^2$ 
to the right-hand side of the condition \eqref{SCon:Dh}.

\medskip
\noindent
{\bf\boldmath 
Additional topic.}
We also construct flat-hump-shaped stationary solutions 
of \eqref{Sys:Main}.
The key assumptions are 
$D(r,s)\equiv D_1(r)D_2(s)$ and $h(r,s)\equiv h_1(r)h_2(s)$.
These assumptions are essential to define $j(u,v)$
such that $\nabla[j(u,v)]=0$ as in the proof 
of Proposition \ref{P2} below, 
which generalizes \cite[Proposition 7]{LW-2005-PNDEA}.
Indeed, let $D(u,v) \nabla u - h(u,v) \nabla v=0$.
Then if $D(r,s)\equiv D_1(r)D_2(s)$ and $h(r,s)\equiv h_1(r)h_2(s)$,
we can obtain
$\frac{D_1(u)}{h_1(u)}\nabla u-\frac{h_2(v)}{D_2(v)}\nabla v=0$.
Thus the first and second terms on the left-hand side 
can be rewritten as 
$\nabla [j_1(u)]$ and $\nabla [j_2(v)]$ 
for some functions $j_1$ and $j_2$, respectively,
and then $j$ is defined as $j=j_1-j_2$.

\medskip
\noindent
{\bf Organization of this paper.}
The remainder of this paper is organized as follows.
In Section \ref{Sec:Prelimi} we show existence 
of global classical solutions to
the nondegenerate version of \eqref{Sys:Main}, which will be used 
as approximate problems. 
Section \ref{Sec:GWSE} is devoted to the the proof of  
existence of global weak solutions 
by passing to the limit of approximate solutions
obtained from Section \ref{Sec:Prelimi}.
Uniqueness of global weak solutions is discussed
in Section \ref{Sec:GWSU}.
In Section \ref{Sec:SS} we consider steady states   
of \eqref{Sys:Main} in the one-dimensional framework.

\section{Preliminaries} \label{Sec:Prelimi}
As a preparation for the proof of Theorem \ref{Thm:exist}, 
we establish existence of global classical solutions to a
nondegenerate version of the model \eqref{Sys:Main}.
\begin{lem} \label{Lem:claS}
Let $p > N$ and let $u_0, v_0 \in W^{1,p}(\Omega)$, 
\red{and suppose that 
\eqref{Con:h} and \eqref{Con:g} hold,
and that
\begin{equation} \label{D:hitaika}
D \in C^2([0,1] \times [0,\infty))
\quad \mbox{and}  
\quad \ 
     \exists\, \eta > 0 \ \mbox{s.t.} \  D(r,s) \ge \eta \quad
     \forall\, (r,s) \in [0,1] \times [0,\infty).
\end{equation}
Then} the problem \eqref{Sys:Main} possesses a
unique global classical solution $(u,v)$ such that
\[
    (u,v) \in \big(C(\overline{\Omega} \times [0,\infty)) \cap
                       C^{2,1}(\overline{\Omega} \times (0,\infty))\big)^2
\] 
and that
\[
    0\le u(x,t)\le 1\quad \mbox{and} \quad v(x,t)\ge 0 \quad 
    \forall\, (x,t)\in \overline{\Omega} \times [0,\infty).
\]
Moreover, we have
\begin{equation} \label{shituryouhozonn2}
    \int_\Omega u(x,t) \, dx = \int_\Omega u_0(x) \, dx \quad 
    \forall\, t \ge 0.
\end{equation}
\end{lem}
\begin{proof}
By virtue of the continuity of $D$ and \eqref{D:hitaika}, 
we see that there exist $\delta > 0$ and
$\overline{D}, \overline{h}, \overline{g}
\in C^2((-\delta,1+\delta) \times (-\delta,\infty))$ such that
\begin{equation}\label{tuikaa}
\overline{D}(r,s)=D(r,s),\quad\overline{h}(r,s)=h(r,s)\quad\mbox{and}\quad
\overline{g}(r,s)=g(r,s)
\end{equation} for all $(r,s) \in [0,1]\times [0,\infty)$,
\red{and that}
\begin{equation} \label{Con:Dba-}
    \overline{D}(r,s) \ge \frac{\eta}{2} \quad 
    \forall\, (r,s) \in (-\delta,1+\delta) \times (-\delta,\infty). 
\end{equation}
\red{
Indeed, 
for a function $D \in C^2([0,1] \times [0,\infty))$, 
by continuity, 
we can take 
$\delta > 0$ such that
\begin{align*}
\begin{cases}
\frac12 D_{rr}(0,0)r^2 + D_r(0,0)r 
+ \frac12 D_{ss}(0,0)s^2 + D_s(0,0)s +D(0,0) 
\ge \frac{\eta}{2}, 
\\
\frac12D_{rr}(1,0)(r-1)^2 + D_r(1,0)(r-1) 
+ \frac12 D_{ss}(1,0)s^2 + D_s(1,0)s +D(1,0)
\ge \frac{\eta}{2}
\end{cases}
\end{align*}
for all 
$(r,s) \in (-\delta,1+\delta) \times (-\delta,\infty)$, 
and  
the $C^2$-extension $\overline{D}$ can be defined as 
\begin{align*}
\overline{D}(r,s) :=
\begin{cases}
D(r,s) & \mbox{for} \  (r,s) \in [0,1] \times [0,\infty), \\
D_1(r,s) 
& \mbox{for} \ (r,s) \in (-\delta, 0] \times (-\delta, 0], \\
D_2(r,s)
& \mbox{for} \ (r,s) \in [1, 1+\delta) \times (-\delta, 0],
\end{cases}
\end{align*}
where 
\begin{align*}
\begin{cases}
D_1(r,s):=
\frac12 D_{rr}(0,0)r^2 + D_r(0,0)r 
+ \frac12 D_{ss}(0,0)s^2 + D_s(0,0)s +D(0,0), 
\\
D_2(r,s):=
\frac12D_{rr}(1,0)(r-1)^2 + D_r(1,0)(r-1) 
+ \frac12 D_{ss}(1,0)s^2 + D_s(1,0)s +D(1,0),
\end{cases}
\end{align*}
and similarly for 
$\overline{h}$ and $\overline{g}$.
}
Set $R_{\delta} := (-\delta,\infty) \times (-\delta,1+\delta)$.
Let
\begin{align*}
    a(y)&:= 
    \begin{pmatrix}
                1       &                        0 
    \\
         -\overline{h}(y_2,y_1) &\hspace{5pt} \overline{D}(y_2,y_1)
    \end{pmatrix}, 
    \\[3mm]
    f(y) &:= 
    \begin{pmatrix}
         \overline{g}(y_2,y_1) 
    \\
               0
    \end{pmatrix}
\end{align*}
for $y = (y_1,y_2) \in R_\delta$. 
Then $a \in C^2(R_{\delta}; \R^{2\times2})$, that is, 
all components of the matrix $a(y)$ belong to $C^2(R_\delta)$.
For $y \in R_\delta$, setting $a_{jk}(y) := a(y) \delta_{jk}$ 
for $1 \le j, k \le N$, where $\delta_{jk}$ is the the Kronecker delta, 
we introduce the operators
\begin{align*}
    \mathcal{A}(y)z&:= -\sum_{j, k=1}^{N} \partial_j(a_{jk}(y) \partial_k z), \\[2mm]
    \mathcal{B}(y)z&:= \sum_{j, k=1}^{N} \nu_j \cdot a_{jk}(y) \partial_kz
\end{align*}
with $z = (z_1,z_2)$, where $\nu=(\nu_1, \dots , \nu_N)$.
Then \eqref{Sys:Main} with $D=\overline{D}$,
$h=\overline{h}$ and $g=\overline{g}$ 
can be rewritten as
\[
    \left\{
    \begin{aligned}
      z_t + \mathcal{A}(z)z &= f(z), & \quad 
      &x\in \Omega, \  t>0,    
    \\[2mm]
      \mathcal{B}(z)z &= 0, & 
      &x\in \partial\Omega, \  t>0, 
    \\[2mm]
      z(0) &= (v_0,u_0), & &x\in \Omega
    \end{aligned}
    \right.
\]
with $z = (v,u)$. 
From \eqref{Con:Dba-} it follows that all eigenvalues of $a(y)$ are positive 
for $y \in R_\delta$. 
Also, $(\mathcal{A}(y), \mathcal{B}(y))$ is of separated divergence form
in the sense of \cite[Example 4.3 (e)]{A-1993-TTM} and hence,
$(\mathcal{A}(y), \mathcal{B}(y))$ is normally elliptic 
for all $y\in R_\delta$ by \cite[Section 4]{A-1993-TTM}.
Therefore, we see from \cite[Theorems 14.4 and 14.6]{A-1993-TTM} that 
\eqref{Sys:Main} with 
$D=\overline{D}$, $h=\overline{h}$ and $g=\overline{g}$
has a unique maximal classical solution
\[
    z = (v,u) \in \big(C(\overline{\Omega} \times [0,T_*)) \cap 
                             C^{2,1}(\overline{\Omega} \times (0,T_*))\big)^2 \quad 
    \text{with}\quad T_* \in (0,\infty].
\]

We now claim that
\begin{equation} \label{Con:bounds}
    0 \le u \le 1 \quad \mbox{and} \quad v \ge 0 \quad 
    \mbox{in} \ \overline{\Omega} \times [0,T_*).
\end{equation}

\red{
We first claim that $u \ge 0$ in $\overline{\Omega} \times [0,T_*)$. 
To this end, we let
$w_+ := \max\{ w, 0 \}$ and let $w_- := - \min\{ w , 0 \}$. 
Then we have $w = w_+ - w_-$.
By multiplying the first equation in \eqref{Sys:Main} by $u_-$
and integrating it over $\Omega$, we observe 
from \eqref{Con:Dba-} and the Young inequality that
\begin{align*}
    \frac{d}{dt} \int_\Omega (u_-)^2 \, dx
    &= -2 \int_\Omega \overline{D}(u,v) |\nabla (u_-)|^2 \, dx 
         -2 \int_\Omega \overline{h}(u_-,v) 
         \nabla v \cdot \nabla (u_-) \, dx 
    \\
    &\le -\eta \int_\Omega |\nabla (u_-)|^2 \, dx 
    -2 \int_\Omega (\overline{h}(u_-,v)-\overline{h}(0,v)) \nabla v \cdot \nabla (u_-) \, dx \\
    &\le \frac{1}{\eta} \int_\Omega |(\overline{h}(u_-,v)-\overline{h}(0,v)) \nabla v|^2 \, dx 
    \\
    &\le c_2 \int_\Omega (u_-)^2 \, dx,
\end{align*}
where $c_2 > 0$ is a constant determined by the mean value theorem and the fact that
$u,v \in C(\overline{\Omega} \times [0,T_*)) \cap 
                             C^{2,1}(\overline{\Omega} \times (0,T_*))$.
Thanks to the Gronwall lemma, we derive 
$\int_\Omega (u_-^2) \, dx = 0$
for all $t \in (0,T_*)$. 
This yields the inequality $u \ge 0$ in $\overline{\Omega} \times [0,T_*)$.
Moreover, noting that $(\widetilde{u}, \widetilde{v}):=(1-u, v)$ 
is a solution of the problem
\begin{align*}
    \begin{cases}
        \widetilde{u}_t 
        = \nabla \cdot (\overline{D}(1-\widetilde{u}, \widetilde{v}) \nabla \widetilde{u} 
                               + \overline{h}(1-\widetilde{u},\widetilde{v}) \nabla \widetilde{v}),
        &  \quad x\in \Omega, \  t>0,
        \\
        \widetilde{v}_t = \Delta \widetilde{v} + \overline{g}(1-\widetilde{u},\widetilde{v}),
        & \quad x\in \Omega, \  t>0,
        \\
        (\overline{D}(1-\widetilde{u},\widetilde{v}) \nabla \widetilde{u} 
         + \overline{h}(1-\widetilde{u},\widetilde{v}) \nabla \widetilde{v}) \cdot \nu 
        = \nabla\widetilde{v} \cdot \nu = 0,
        &  \quad x\in \partial\Omega, \  t>0,
        \\
        \widetilde{u}(\cdot, 0) = 1-u_0, \quad \widetilde{v}(\cdot, 0) = v_0,
        &  \quad x\in \Omega,
    \end{cases}
\end{align*}
we conclude that 
$u(x,t) \le 1$ for all $(x,t) \in \overline{\Omega} \times [0,T_*)$.

Next, we show that $v \ge 0$ in $\overline{\Omega} \times [0,T_*)$. 
Multiplying the second equation in \eqref{Sys:Main} by $v_-$, 
integrating it over $\Omega$ and 
using the boundary condition for $v$ in \eqref{Sys:Main}, 
we infer that 
\begin{align*}
    \frac{d}{dt} \int_\Omega (v_-)^2 \, dx 
    &= -2 \int_\Omega |\nabla v_-|^2 \, dx 
         -2 \int_\Omega \overline{g}(u,v) v_- \, dx 
    \\
    &\le -2 \int_\Omega (\overline{g}(u,v)-\overline{g}(u,0))v_- \, dx 
    \\
    &\le c_1 \int_\Omega (v_-)^2 \, dx,
\end{align*}
where $c_1>0$ is a constant determined by 
the mean value theorem and the fact that
$u,v \in C(\overline{\Omega} \times [0,T_*)) \cap 
                             C^{2,1}(\overline{\Omega} \times (0,T_*))$.
The Gronwall lemma implies that $\int_\Omega (v_-)^2 \, dx = 0$ 
for all $t \in (0,T_*)$, which means $v_- = 0$. 
Thus $v = v_+ - v_- = v_+ \ge 0$ in $\overline{\Omega} \times [0,T_*)$.}

Also, the identity \eqref{shituryouhozonn2} follows by integrating the first equation over
$\Omega \times (0,t)$.

Finally, by the second equation in \eqref{Sys:Main} 
and the condition \eqref{Con:g} together with \eqref{Con:bounds}, we have
\begin{align*}
    v_t(x,t) - \Delta v(x,t)
    &\le g(u(x,t),0) + \kappa v(x,t) 
    \\
    &\le (\| g(\cdot,0) \|_{L^\infty(0,1)} + \kappa) (1 + v(x,t))
\end{align*}
for all $(x,t) \in \overline{\Omega} \times [0,T_*)$. 
Thus
\begin{equation} \label{ini:vupbound}
    v(x,t) \le (1 + \| v_0 \|_{L^\infty(\Omega)}) e^{c_3 t} \quad 
    \forall\,(x,t) \in \overline{\Omega} \times [0,T_*),
\end{equation}
where $c_3 := \| g(\cdot,0) \|_{L^\infty(0,1)} + \kappa$. 
By virtue of \cite[Theorem 15.5]{A-1993-TTM}, 
we see that $T_* = \infty$.
Since \eqref{tuikaa} and \eqref{Con:bounds} hold, 
we arrive at the conclusion of Lemma \ref{Lem:claS}.
\end{proof}
\section{Existence of global weak solutions} \label{Sec:GWSE}
\begin{proof}[Proof of Theorem \ref{Thm:exist} (Existence)] 
The proof will be achieved through six steps. 

\medskip
\noindent
{\bf Step 1. Construction of approximate solutions.}
Fix $\ep \in (0,1)$ and let
\begin{equation} \label{Dep}
D^\ep(r,s) := D(r,s) + \ep, \quad (r,s) \in [0,1] \times [0,\infty).
\end{equation}
Also, in view of \eqref{Con:ini}, we can take 
$(u_0^\ep,v_0^\ep) \in \big(W^{1,N+1}(\Omega)\cap L^\infty(\Omega)\big)^2$ 
such that $0 \le u_0^\ep \le 1$ and $v_0^\ep \ge 0$ 
with $v_0^\ep \in H^2(\Omega)$ and 
$\nabla v_0^\ep \cdot \nu = 0$ on $\partial \Omega$ 
and  that
\begin{equation} \label{inikinji}
\begin{cases}
\| v_0^\ep \|_{L^\infty(\Omega)} \le \| v_0 \|_{L^\infty(\Omega)}+1 
\quad \mbox{and} \quad 
\| v_0^\ep \|_{H^{2}(\Omega)} \le \| v_0 \|_{H^{2}(\Omega)}+1, 
\\
\| u_0^\ep - u_0 \|_{L^2(\Omega)} + \| v_0^\ep - v_0 \|_{L^2(\Omega)} \le \ep.
\end{cases}
\end{equation}
According to Lemma \ref{Lem:claS} with $D$ and $(u_0,v_0)$ 
replaced by $D^\ep$ and $(u^\ep_0, v^\ep_0)$, respectively, 
the problem \eqref{Sys:Main} admits a unique classical solution 
$(u^\ep,v^\ep)$ fulfilling
\begin{align}\label{Sys:Mainkinji}
    \begin{cases}
        (u^\ep)_t 
        = \nabla \cdot (D^\ep(u^\ep,v^\ep) \nabla u^\ep 
                               - h(u^\ep,v^\ep) \nabla v^\ep),
        & \quad x\in \Omega, \  t>0,
        \\
        (v^\ep)_t = \Delta v^\ep + g(u^\ep,v^\ep),
        & \quad x\in \Omega, \  t>0,
        \\
        (D^\ep(u^\ep,v^\ep) \nabla u^\ep 
         - h(u^\ep,v^\ep) \nabla v^\ep) \cdot \nu 
        = \nabla v^\ep \cdot \nu = 0,
        & \quad x\in \partial \Omega, \  t>0,
        \\
        u^\ep(\cdot, 0) = u_0^\ep, \quad v^\ep(\cdot, 0) = v_0^\ep,
        & \quad x\in \Omega,
    \end{cases}
\end{align}
and
\begin{equation} \label{clabounds}
0 \le u^\ep(x,t) \le 1 \quad \mbox{and} \quad v^\ep(x,t) \ge 0 \quad
\forall\,(x,t) \in \overline{\Omega} \times [0,\infty).
\end{equation}
Moreover, a combination of \eqref{ini:vupbound} with \eqref{inikinji} yields
\begin{equation*} 
0 \le v^\ep(x,t) \le (2 + \| v_0 \|_{L^\infty(\Omega)}) e^{c_3 t} \quad 
\forall\, (x,t) \in \overline{\Omega} \times [0,\infty).
\end{equation*}
In particular, for each $T > 0$, there exists $C_1(T)>0$ independent of $\ep$ 
such that
\begin{equation} \label{vepbounds}
0 \le v^\ep(x,t) \le C_1(T) \quad 
\forall\, (x,t) \in \overline{\Omega} \times [0,T].
\end{equation}

\medskip
\noindent
{\bf \boldmath Step 2. $\ep$-independent estimates for 
$\Delta v^\ep$ and $(v^\ep)_t$.}
Fix $\ep \in (0,1)$. 
We claim that 
for any  $T>0$ there exists $C_2(T) > 0$ such that
\begin{equation} \label{vepesti}
\int_0^T \| v^\ep \|_{L^2(\Omega)}^2 \, dt 
+\int_0^T \| \Delta v^\ep \|_{L^2(\Omega)}^2 \, dt 
+ \int_0^T \| (v^\ep)_t \|_{L^2(\Omega)}^2 \, dt
\le C_2(T),
\end{equation}
in particular, since 
$\|\nabla v^\ep \|_{L^2(\Omega)}^2
 =(v^\ep, -\Delta v^\ep)_{L^2(\Omega)} 
 \le \| v^\ep \|_{L^2(\Omega)} \|\Delta v^\ep \|_{L^2(\Omega)}$,
we have
\begin{equation} \label{tuikadayo}
\int_0^t \|\nabla v^\ep \|_{L^2(\Omega)}^2 \, dt \le C_2(T),
\end{equation}
where $C_2(T)$ is a constant depending on $T$
and is independent of $\ep$.
Since the inhomogeneous term $g(u^\ep,v^\ep)$ in \eqref{Sys:Mainkinji} 
is bounded in $\Omega \times [0,T)$ for all $T>0$ 
by virtue of \eqref{clabounds} and \eqref{vepbounds},
when combined with \eqref{inikinji}, \eqref{vepbounds} and 
the maximal Sobolev regularity 
for parabolic equations \cite[3.1 Theorem]{HP-1997-CPDE}, 
the inequality \eqref{vepesti} follows.

\medskip
\noindent
{\bf \boldmath Step 3. $\ep$-independent estimates for 
$D^\ep(u^\ep,v^\ep)\nabla u^\ep$.}
We intend to confirm that 
\begin{equation} \label{Con:Depuepvep}
\int_0^T \| D^\ep(u^\ep,v^\ep) \nabla u^\ep \|_{L^2(\Omega)}^2 \, dt \le C(T)
\end{equation}
for some $C(T)>0$ which is independent of $\ep$.
To this end, for each $r\in [0,1]$ and $s\in [0,\infty)$, put
\begin{align}
\label{hanaDep}
\mathcal{D}^\ep(r,s) 
&:= \int_0^r D^\ep (\sigma,s) \, d\sigma, 
\\
\label{Con:hanaDeptil}
\widetilde{\mathcal{D}^\ep}(r,s) 
&:= \int_0^r \mathcal{D}^\ep (\sigma,s) \, d\sigma.
\end{align}
Then we can see that $\mathcal{D}^\ep \in C^1([0,1] \times [0,\infty))$ and
$\widetilde{\mathcal{D}^\ep} \in C^2([0,1] \times [0,\infty))$ with
\[
(\widetilde{\mathcal{D}^\ep})_{rr} (r,s) = (\mathcal{D}^\ep)_r (r,s) = D^\ep (r,s)
\]
for all $(r,s) \in [0,1] \times [0,\infty)$ and with
$\widetilde{\mathcal{D}^\ep} (0,s) = \mathcal{D}^\ep (0,s) = 0$
for all $s \in [0,\infty)$.
Using the first and second equations in \eqref{Sys:Mainkinji}, we have
\begin{align}
\nonumber
\frac{d}{dt} \int_\Omega \widetilde{\mathcal{D}^\ep}(u^\ep,v^\ep) \, dx
&=\int_\Omega (\widetilde{\mathcal{D}^\ep})_r(u^\ep,v^\ep)
                       \cdot(u^\ep)_t \, dx
    +\int_\Omega (\widetilde{\mathcal{D}^\ep})_s(u^\ep,v^\ep)
                        \cdot(v^\ep)_t \, dx 
\\
\nonumber
&=\int_\Omega 
          \mathcal{D}^\ep(u^\ep,v^\ep) 
          \nabla \cdot (D^\ep(u^\ep,v^\ep) \nabla u^\ep 
                                     - h(u^\ep,v^\ep) \nabla v^\ep)
    \, dx 
\\
\nonumber
&\quad \, +\int_\Omega (\widetilde{\mathcal{D}^\ep})_s(u^\ep,v^\ep)
\big(\Delta v^\ep + g(u^\ep,v^\ep)\big) \, dx \\[2mm]
\label{tuika} 
&=:\mathcal{I}_1+\mathcal{I}_2.
\end{align}
We consider the first term $\mathcal{I}_1$ 
on the right-hand side of \eqref{tuika}.
Integration by parts gives
\begin{align}
\nonumber
\mathcal{I}_1
&= - \int_\Omega \nabla [\mathcal{D}^\ep(u^\ep,v^\ep)] \cdot
(D^\ep(u^\ep,v^\ep) \nabla u^\ep - h(u^\ep,v^\ep) \nabla v^\ep) \, dx \\
\nonumber
&=- \int_\Omega (D^\ep(u^\ep,v^\ep) \nabla u^\ep 
    + (\mathcal{D}^\ep)_s(u^\ep,v^\ep) \nabla v^\ep) \cdot 
    (D^\ep(u^\ep,v^\ep) \nabla u^\ep 
- h(u^\ep,v^\ep) \nabla v^\ep) \, dx \\
\nonumber
&
\le - \int_\Omega D^\ep(u^\ep,v^\ep)^2 |\nabla u^\ep|^2 \, dx
+\int_\Omega D^\ep(u^\ep,v^\ep) h(u^\ep,v^\ep) |\nabla u^\ep| |\nabla v^\ep| \, dx
 \\
\nonumber
&
\quad \, +\int_\Omega (\mathcal{D}^\ep)_s (u^\ep,v^\ep)
D^\ep(u^\ep,v^\ep) |\nabla u^\ep| |\nabla v^\ep| \, dx
+\int_\Omega(\mathcal{D}^\ep)_s(u^\ep,v^\ep)h(u^\ep,v^\ep) |\nabla v^\ep|^2 \, dx 
\\[2mm]
\label{X}
&
=: \mathcal{J}_1 + \mathcal{J}_2 + \mathcal{J}_3 + \mathcal{J}_4. 
\end{align}
Here, we have
\begin{equation} \label{j1st}
\mathcal{J}_1=-\| D^\ep(u^\ep,v^\ep) \nabla u^\ep \|_{L^2(\Omega)}^2.
\end{equation}
Recalling \eqref{vepbounds} and letting $\delta>0$, 
we see from the Young inequality that 
\begin{equation} \label{Con:ap1stright2nd}
 \mathcal{J}_2
 \le \| h \|_{L^\infty((0,1) \times (0,C_1(T)))} 
 \Big(\delta \| D^\ep(u^\ep,v^\ep) \nabla u^\ep \|_{L^2(\Omega)}^2
  + \frac{1}{4\delta} \| \nabla v^\ep \|_{L^2(\Omega)}^2\Big).
\end{equation}
In order to estimate $\mathcal{J}_3$ and $\mathcal{J}_4$ we check that
\begin{equation} \label{Con:hanaDeppas}
| (\mathcal{D}^\ep)_s(u^\ep,v^\ep) | \le M(T)
\end{equation}
for some $M(T) > 0$ which is independent of $\ep$. 
Indeed, by the definitions of $\mathcal{D}^\ep$, $D^\ep$ and 
$\mathcal{D}$ (see \eqref{Dep}, 
\eqref{hanaDep} and \eqref{def:hanaD}), 
for each $r\in[0,1]$ and $s\in[0,\infty)$,
we have
\[
    \mathcal{D}^\ep(r,s) 
    = \int_0^r D^\ep(\sigma,s) \, d\sigma 
    = \int_0^r D(\sigma,s) \, d\sigma + \ep r
    = \mathcal{D}(r,s) + \ep r.
\]
Differentiating this identity with respect to $s$, we obtain 
$(\mathcal{D}^\ep)_s(r,s) = \mathcal{D}_s(r,s)$,
which along with \eqref{clabounds}, \eqref{vepbounds} 
and 
\red{$D \in C^2([0,1] \times [0, K])$}
with $K = C_1(T)$ leads to \eqref{Con:hanaDeppas}. 
We now estimate $\mathcal{J}_3$ and $\mathcal{J}_4$ 
on the right-hand side of \eqref{X}.
By virtue of \eqref{Con:hanaDeppas} and the Young inequality, we observe that 
\begin{equation} \label{Con:ap1stright3rd}
\mathcal{J}_3
\le \delta M(T)  \| D^\ep(u^\ep,v^\ep) \nabla u^\ep \|_{L^2(\Omega)}^2
 + \frac{1}{4\delta} M(T) \| \nabla v^\ep \|_{L^2(\Omega)}^2,
\end{equation}
and use \eqref{clabounds}, \eqref{vepbounds} and \eqref{Con:hanaDeppas} to reveal that 
\begin{equation} \label{Con:ap1stright4th}
\mathcal{J}_4
\le M(T) \| h \|_{L^\infty((0,1) \times (0,C_1(T)))} 
             \| \nabla v^\ep \|_{L^2(\Omega)}^2.
\end{equation}
Collecting \eqref{j1st}, \eqref{Con:ap1stright2nd}, \eqref{Con:ap1stright3rd} 
and \eqref{Con:ap1stright4th} in \eqref{X} yields
\begin{align*}
\mathcal{I}_1 
&\le (\delta \| h \|_{L^\infty((0,1) \times (0,C_1(T)))} + \delta M(T) - 1)
\| D^\ep(u^\ep,v^\ep)\nabla u^\ep \|_{L^2(\Omega)}^2 \\
&\quad \, + 
\Big(\frac{1}{4 \delta} \| h \|_{L^\infty((0,1) \times (0,C_1(T)))}
  + \frac{1}{4 \delta} M(T) + M(T) \| h \|_{L^\infty((0,1) \times (0,C_1(T)))}\Big)
  \| \nabla v^\ep \|_{L^2(\Omega)}^2.
\end{align*}
Here, fixing 
$0 < \delta < \frac{1}{\| h \|_{L^\infty((0,1) \times (0,C_1(T)))} + M(T)}$ 
and setting
\begin{align*}
&C_3(T) := 1-\delta \| h \|_{L^\infty((0,1) \times (0,C_1(T)))} - \delta M(T), \\
&C_4(T) := \frac{1}{4 \delta} \| h \|_{L^\infty((0,1) \times (0,C_1(T)))}
  + \frac{1}{4 \delta} M(T) + M(T) \| h \|_{L^\infty((0,1) \times (0,C_1(T)))},
\end{align*}
we have
\begin{equation}\label{Con:ap1strightcon}
\mathcal{I}_1
\le -C_3(T)\| D^\ep(u^\ep,v^\ep)\nabla u^\ep \|_{L^2(\Omega)}^2 \\
+ C_4(T) \| \nabla v^\ep \|_{L^2(\Omega)}^2.
\end{equation}
We consider the second term $\mathcal{I}_2$ 
on the right-hand side of \eqref{tuika}.
Due to the definition of $\widetilde{\mathcal{D}^\ep}$
(see \eqref{Con:hanaDeptil}) and \eqref{Con:hanaDeppas}, we obtain
$| (\widetilde{\mathcal{D}^\ep})_s(u^\ep,v^\ep) | \le M(T).$
This fact together with the H\"{o}lder inequality, \eqref{clabounds} 
and \eqref{vepbounds} entails that
\begin{align}
\nonumber
\mathcal{I}_2
&\le M(T)\int_\Omega | \Delta v^\ep | \, dx
+ M(T) \int_\Omega |g(u^\ep,v^\ep)| \, dx \\
\label{Con:ap1stconchanright3rd}
&\le \frac12 M(T)|\Omega|^{\frac12} + \frac12 M(T) \| \Delta v^\ep \|_{L^2(\Omega)}^2
+ M(T) |\Omega| \| g \|_{L^\infty((0,1) \times (0,C_1(T)))}.
\end{align}
Combining \eqref{Con:ap1strightcon} 
and \eqref{Con:ap1stconchanright3rd} with \eqref{tuika},
we see that
\begin{align*}
\nonumber
&\frac{d}{dt} \int_\Omega \widetilde{\mathcal{D}^\ep}(u^\ep,v^\ep) \, dx
+C_3(T)\| D^\ep(u^\ep,v^\ep) \nabla u^\ep \|_{L^2(\Omega)}^2 \\
\nonumber
&\le C_4(T) 
  \| \nabla v^\ep \|_{L^2(\Omega)}^2 
  + \frac12 M(T)|\Omega|^{\frac12}  
+ \frac12 M(T) \| \Delta v^\ep \|_{L^2(\Omega)}^2
+ M(T) |\Omega| \| g \|_{L^\infty((0,1) \times (0,C_1(T)))}.
\end{align*}
By integrating this inequality over $(0,T)$, 
the estimates \eqref{vepesti} and \eqref{tuikadayo} yield
\begin{align*}
&C_3(T) \int_0^T \| D^\ep(u^\ep,v^\ep) \nabla u^\ep \|_{L^2(\Omega)}^2 \, dt 
+\int_\Omega \widetilde{\mathcal{D}^\ep}(u^\ep(x,T),v^\ep(x,T)) \, dx \\
&\le \int_\Omega \widetilde{\mathcal{D}^\ep}(u_0^\ep(x),v_0^\ep(x)) \, dx
 + \Big(C_4(T) +\frac12 M(T)\Big) C_2(T) \\
 &\quad \, 
 + T M(T)\Big(\frac12 |\Omega|^{\frac12} 
 + |\Omega| \| g \|_{L^\infty((0,1) \times (0,C_1(T)))}\Big).
\end{align*}
Here, we note that there is $C>0$ such that
$\int_\Omega \widetilde{\mathcal{D}^\ep}(u_0^\ep(x),v_0^\ep(x)) \, dx \le C$
in view of \eqref{inikinji}.
Therefore, we conclude as intended.
 
\medskip
\noindent
{\bf \boldmath Step 4. $\ep$-independent estimates for 
$\nabla[\mathcal{D}^\ep(u^\ep,v^\ep)]$, $\nabla[\mathcal{D}_0(u^\ep)]$
and $(u^\ep)_t$. }
For each $r\in [0,1]$, we let a function $\mathcal{D}_0$ defined by 
\begin{equation} \label{hanaD0}
\mathcal{D}_0(r):=\int_0^r D_0(\sigma) \, d\sigma.
\end{equation}
Then we assert that 
\begin{align} \label{Con:con}
&\int_0^T \| \nabla [\mathcal{D}^\ep(u^\ep,v^\ep)] \|_{L^2(\Omega)}^2 \, dt
\le C(T), \\
 \label{Con:D0}
&\int_0^T  \| \nabla [\mathcal{D}_0(u^\ep)] \|_{L^2(\Omega)}^2 \, dt \le C(T), \\
 \label{Con:ueppat}
&\int_0^T  \| (u^\ep)_t \|_{(H^1(\Omega))'}^2 \, dt \le C(T)
\end{align}
for some $C(T)>0$ (not relabelled) which is independent of $\ep$.
First, by \eqref{Con:hanaDeppas}, we have 
\begin{align*}
| \nabla [\mathcal{D}^\ep(u^\ep,v^\ep)] | 
&\le | D^\ep(u^\ep,v^\ep) \nabla u^\ep |
+ | (\mathcal{D}^\ep)_s(u^\ep,v^\ep) | | \nabla v^\ep | \\
&\le | D^\ep(u^\ep,v^\ep) \nabla u^\ep | + M(T) | \nabla v^\ep |.
\end{align*}
Hence we make sure that
\[
\| \nabla [\mathcal{D}^\ep(u^\ep(t),v^\ep(t))] \|_{L^2(\Omega)}^2 
\le 2 \| D^\ep(u^\ep(t),v^\ep(t)) \nabla u^\ep(t) \|_{L^2(\Omega)}^2
+ 2 M(T)^2 \| \nabla v^\ep(t) \|_{L^2(\Omega)}^2.
\]
By integrating this inequality over $(0,T)$, 
it follows from \eqref{tuikadayo} and \eqref{Con:Depuepvep} 
that 
\begin{align*}
\int_0^T \| \nabla [\mathcal{D}^\ep(u^\ep,v^\ep)] \|_{L^2(\Omega)}^2 \, dt 
&\le 2 \int_0^T \| D^\ep(u^\ep,v^\ep) \nabla u^\ep \|_{L^2(\Omega)}^2 \, dt
+ 2 M(T)^2 \int_0^T \| \nabla v^\ep \|_{L^{2}(\Omega)}^2 \, dt \\
&\le 2 C(T) + 2 M(T)^2 C_2(T).
\end{align*}
Thus we obtain \eqref{Con:con}.
Secondly, since $D^\ep(u^\ep,v^\ep) \ge D(u^\ep,v^\ep) \ge D_0(u^\ep)$ by
\eqref{Con:D}, we have
$| D^\ep(u^\ep,v^\ep) \nabla u^\ep| \ge | D_0(u^\ep) \nabla u^\ep | 
= | \nabla [\mathcal{D}_0(u^\ep)] |$
due to \eqref{hanaD0}, 
and hence, \eqref{Con:Depuepvep} yields \eqref{Con:D0}.
Finally, it follows from
the first equation in \eqref{Sys:Mainkinji} 
and the definition of $\mathcal{D}^\ep$ (see \eqref{hanaDep}) that
\begin{align*}
&\| (u^\ep)_t \|_{(H^1(\Omega))'}^2\\ 
&= \sup_{\substack{\varphi \in H^1(\Omega) \\ \| \varphi \|_{H^1(\Omega)} \le 1}}
| \langle (u^\ep)_t,\varphi \rangle_{(H^1(\Omega))',H^1(\Omega)} |^2 \\
&= \sup_{\substack{\varphi \in H^1(\Omega) \\ \| \varphi \|_{H^1(\Omega)} \le 1}}
| \langle D^\ep(u^\ep,v^\ep) \nabla u^\ep 
- h(u^\ep,v^\ep) \nabla v^\ep,\nabla \varphi \rangle_{(H^1(\Omega))',H^1(\Omega)} |^2 \\
&= \sup_{\substack{\varphi \in H^1(\Omega) \\ \| \varphi \|_{H^1(\Omega)} \le 1}}
| \langle \nabla [\mathcal{D}^\ep(u^\ep,v^\ep)]  
- ((\mathcal{D}^\ep)_s(u^\ep,v^\ep) + h(u^\ep,v^\ep)) \nabla v^\ep,
\nabla \varphi \rangle_{(H^1(\Omega))',H^1(\Omega)} |^2.
\end{align*}
Also, \eqref{clabounds}, \eqref{vepbounds} and \eqref{Con:hanaDeppas}
give the inequality
\[
    |(\mathcal{D}^\ep)_s(u^\ep,v^\ep) + h(u^\ep,v^\ep)| 
    \le M(T) + \| h \|_{L^\infty((0,1) \times (0,C_1(T)))}.
\]
Therefore, we have
\begin{align*}
\| (u^\ep)_t \|_{(H^1(\Omega))'}^2
&\le \sup_{\substack{\varphi \in H^1(\Omega) \\ \| \varphi \|_{H^1(\Omega)} \le 1}}
\Big(\| \nabla [\mathcal{D}^\ep(u^\ep,v^\ep)] \|_{L^2(\Omega)} 
 \| \nabla \varphi \|_{L^2(\Omega)} \\
&\hspace{25mm}+ (M(T) + \| h \|_{L^\infty((0,1) \times (0,C_1(T)))}) 
  \| \nabla v^\ep \|_{L^2(\Omega)} \| \nabla \varphi \|_{L^2(\Omega)}\Big)^2 \\
&\le 2 \Big(\| \nabla [\mathcal{D}^\ep(u^\ep,v^\ep)] \|_{L^2(\Omega)}^2
+ \big(M(T) + \| h \|_{L^\infty((0,1) \times (0,C_1(T)))}\big)^2
  \| \nabla v^\ep \|_{L^2(\Omega)}^2\Big).
\end{align*}
Integrating this inequality over $(0,T)$ and using \eqref{tuikadayo} 
and \eqref{Con:con}, we obtain 
\[
\int_0^T \| (u^\ep)_t \|_{(H^1(\Omega))'}^2 \, dt 
\le 2 \Big(C(T) + \big(M(T) + \| h \|_{L^\infty((0,1) \times (0,C_1(T)))}\big)^2C_2(T)\Big).
\]
Therefore, \eqref{Con:ueppat} holds.

\medskip
\noindent
{\bf \boldmath Step 5. Construction of a limit $(u,v)$.}
We define a function $Q\in C^2([0,1])$ by setting 
\[
Q(r) :=\int_0^r D_0^2(\sigma) \, d\sigma \quad \mbox{for} \ r \in [0,1].
\]
Now, we claim that for each $T > 0$,
\begin{equation} \label{Puep}
\left\{ Q(u^\ep) \right\}_{\ep \in (0,1)} \, \mbox{is bounded in} \,
\{ w \in L^2(0,T;H^1(\Omega)) \mid 
w_t \in L^1(0,T;(W^{1,N+1}(\Omega))') \}.
\end{equation}
Fix $T>0$.
First, we verify that $Q(u^\ep) \in L^2(0,T;H^1(\Omega))$.
We deduce from 
\red{\eqref{clabounds} that} 
\[
\| Q(u^\ep) \|_{L^\infty(\Omega \times (0,T))} \le \| Q \|_{L^\infty(0,1)}
\] 
\red{and from 
\eqref{Con:D0} that} there exists $C_5(T)>0$ independent of $\ep$ such that
\begin{align*}
\int_0^T \| \nabla [Q(u^\ep)] \|_{L^2(\Omega)}^2 \, dt
&\le \| D_0 \|_{L^\infty(0,1)}^2
\int_0^T \| \nabla [\mathcal{D}_0(u^\ep)] \|_{L^2(\Omega)}^2 \, dt \\
&\le C_5(T).
\end{align*}
Next, we prove that
$(Q(u^\ep))_t\in L^1(0,T;(W^{1,N+1}(\Omega))')$.
For all $t \in (0,T)$,
we compute
\begin{align}
\nonumber
\|  (Q(u^\ep))_t(t) \|_{(W^{1,N+1}(\Omega))'} 
\nonumber
&=\sup_{\substack{\varphi \in W^{1,N+1}(\Omega) \\ \| \varphi \|_{W^{1,N+1}(\Omega)}\le 1}}
| \langle (Q(u^\ep))_t, \varphi \rangle_{(W^{1,N+1}(\Omega))',W^{1,N+1}(\Omega)} | \\
\nonumber
&\red{{}=\sup_{\substack{\varphi \in W^{1,N+1}(\Omega) \\ \| \varphi \|_{W^{1,N+1}(\Omega)}\le 1}}
| \big( (u^\ep)_t, D_0(u^\ep)^2 \varphi \big)_{L^2(\Omega)} |} \\
\nonumber
&=\sup_{\substack{\varphi \in W^{1,N+1}(\Omega) \\ \| \varphi \|_{W^{1,N+1}(\Omega)}\le 1}}
| \langle (u^\ep)_t, D_0(u^\ep)^2 \varphi \rangle_{(H^{1}(\Omega))',H^{1}(\Omega)} | \\
\label{A}
&\red{{}\le}\sup_{\substack{\varphi \in W^{1,N+1}(\Omega) \\ \| \varphi \|_{W^{1,N+1}(\Omega)}\le 1}}
\| (u^\ep)_t \|_{(H^1(\Omega))'} \| D_0(u^\ep)^2 \varphi \|_{H^1(\Omega)}.
\end{align}
Here, we estimate the term $\| D_0(u^\ep)^2 \varphi \|_{H^1(\Omega)}$
by fixing $\varphi \in W^{1,N+1}(\Omega)$.
Noting that $D_0(u^\ep) \le \| D_0 \|_{L^\infty(0,1)}$ and
$D_0'(u^\ep) \le \| D_0' \|_{L^\infty(0,1)}$ 
by \eqref{clabounds} and that
$\varphi \in L^\infty(\Omega)$ 
due to the continuous embedding of $W^{1,N+1}(\Omega)$ 
in $L^\infty(\Omega)$, we infer that
\begin{align*}
&\| D_0(u^\ep)^2 \varphi \|_{H^1(\Omega)}^2 \\
&\le \|D_0(u^\ep)^2 \varphi\|_{L^2(\Omega)}^2  
      + \|2D_0'(u^\ep) D_0(u^\ep)(\nabla u^\ep)\varphi 
           +D_0(u^\ep)^2 \nabla \varphi\|_{L^2(\Omega)}^2 \\
&\le \| D_0 \|_{L^\infty(0,1)}^4\| \varphi \|_{L^2(\Omega)}^2
+2\|2D_0'(u^\ep) \varphi D_0(u^\ep)\nabla u^\ep \|_{L^2(\Omega)}^2  
+2\|D_0(u^\ep)^2 \nabla \varphi\|_{L^2(\Omega)}^2  \\
&\le 2\Big(\| D_0 \|_{L^\infty(0,1)}^4\| \varphi \|_{H^1(\Omega)}^2
+4\| D_0' \|_{L^\infty(0,1)}^2 \| \varphi \|_{L^\infty(\Omega)}^2 
    \| \nabla [\mathcal{D}_0(u^\ep)] \|_{L^2(\Omega)}^2 \Big).
\end{align*}
Thus we have 
\[
\| D_0(u^\ep)^2 \varphi \|_{H^1(\Omega)} \le \sqrt2
\Big(\| D_0 \|_{L^\infty(0,1)}^2\| \varphi \|_{H^1(\Omega)}
+2\| D_0' \|_{L^\infty(0,1)} \| \varphi \|_{L^\infty(\Omega)} 
    \| \nabla [\mathcal{D}_0(u^\ep)] \|_{L^2(\Omega)} \Big),
\]
which along with the continuous embedding of $W^{1,N+1}(\Omega)$ 
in $L^\infty(\Omega)$ and in $H^1(\Omega)$
yields 
\[
\| D_0(u^\ep)^2 \varphi \|_{H^1(\Omega)} 
\le C_6\| \varphi \|_{W^{1,N+1}(\Omega)}
            \big(1+\| \nabla [\mathcal{D}_0(u^\ep)] \|_{L^2(\Omega)}\big) 
\]
for some $C_6>0$.
Plugging this inequality into \eqref{A}, we obtain
\[
\|  (Q(u^\ep))_t(t) \|_{(W^{1,N+1}(\Omega))'} 
\le C_6 \| (u^\ep)_t \|_{(H^1(\Omega))'}
             \big(1+\| \nabla [\mathcal{D}_0(u^\ep)] \|_{L^2(\Omega)}\big).
\]
By integration of this inequality over $(0,T)$, we arrive at 
the claim \eqref{Puep} by virtue of 
\red{\eqref{Con:ueppat}.}
Once \eqref{Puep} is established, the Aubin--Lions theorem entails 
the relative compactness of 
$\left\{Q(u^\ep) \right\}_{\ep \in (0,1)}$ in $L^2(\Omega \times (0,T))$. 
\red{
Moreover, $Q$ is strictly increasing 
since $Q'(r) = D_0^2(r) > 0$ for $r \in [0,1)$
due to \eqref{Con:D}, 
and hence $Q^{-1}$ 
is continuous 
on the compact interval $Q[0,1]$. 
We thereby obtain that  
$\left\{u^\ep \right\}_{\ep \in (0,1)}$ is relatively compact in 
$L^2(\Omega \times (0,T))$ 
since 
$u^\ep \in [0,1]$ and 
$\left\{Q(u^\ep) \right\}_{\ep \in (0,1)}$
is relatively compact 
in $L^2(\Omega \times (0,T))$.
We now show that  
$\left\{u^\ep\right\}_{\ep \in (0,1)}$ is relatively compact in
$C([0,T];(H^1(\Omega))')$. 
To this end, 
we first show the boundedness of 
$\left\{u^\ep\right\}_{\ep \in (0,1)}$ 
in $C([0,T];(H^1(\Omega))')$. 
The boundedness is proved 
due to the facts that 
$u^\ep \in [0,1]$ 
and that
$u^\ep \in C([0,T];L^2(\Omega))$ 
together with 
the fact that 
the embedding of $L^2(\Omega)$ 
in $(H^1(\Omega))'$ 
is compact. 
Secondly, 
we prove the equicontinuity 
of $\left\{u^\ep\right\}_{\ep \in (0,1)}$ 
in $(H^1(\Omega))'$. 
For all $t,s \in [0,T]$ 
we see from the H\"older inequality that  
\begin{align*}
\| u^\ep(t) - u^\ep(s) \|_{(H^1(\Omega))'}
&= \Big\| \int_{s}^{t} 
   (u^{\ep})_t(r) \, dr \Big\|_{(H^1(\Omega))'}
\\
&\le \Big| \int_{s}^{t} 
   \| (u^{\ep})_t(r) \|^2
              _{(H^1(\Omega))'} \, dr 
                               \Big|^{\frac12}           
\Big| \int_{s}^{t} 1 \, dr \Big|^{\frac12}                                   
\\
&\le \Big( 
              \int_{0}^{T} 
             \| (u^{\ep})_t(r) \|^2
                        _{(H^1(\Omega))'} \, dr 
                                          \Big)^{\frac12}
       | t-s |^{\frac12}
\\
&\le C(T)^{\frac12} | t-s |^{\frac12},                                          
\end{align*} 
where $C(T)$ is the positive constant 
in \eqref{Con:ueppat}. 
Therefore, 
in light of the Arzel\`{a}--Ascoli theorem, 
$\left\{u^\ep\right\}_{\ep \in (0,1)}$ is relatively compact in
$C([0,T];(H^1(\Omega))')$.
Also, \eqref{vepesti} implies that
$\left\{v^\ep \right\}_{\ep \in (0,1)}$ is relatively compact 
in $C([0,T];L^2(\Omega))$.
Invoking \eqref{clabounds} and 
\eqref{vepesti}, 
we thereby conclude that} 
there exist a sequence 
$\{\ep_k\}_{k\in\N}\subset(0,1)$ 
with $\ep_k \to 0$ 
as $k\to \infty$ 
and a couple 
$(u,v) \in L^\infty(\Omega \times (0,\infty)) 
              \times L^\infty_{\rm loc} (\overline{\Omega} \times [0,\infty))$ 
such that
\begin{align} \label{shuusoku1}
&(u^{\ep_k},v^{\ep_k}) \to (u,v) \quad \mbox{in} \ 
L^p(\Omega \times (0,T))\times L^2(\Omega \times (0,T)), \\
\label{shuusoku2}
&(u^{\ep_k},v^{\ep_k}) \to (u,v) \quad \mbox{in} \ 
C([0,T];(H^1(\Omega))') \times C([0,T];L^2(\Omega))
\end{align}
as $k \to \infty$, for each $p \in [1,\infty)$ and $T > 0$.

\medskip
\noindent
{\bf\boldmath Step 6. Conclusion.} We verify that the limit couple 
$(u,v)$ is a solution to the problem \eqref{Sys:Main} in the sense of 
Definition \ref{Def:WS}.
From \eqref{clabounds} and \eqref{shuusoku1} we can show that
$u(x,t)\in [0,1]$ and $v(x,t) \in [0,\infty)$ 
a.e.\ in $(x,t) \in \Omega \times (0,\infty)$. 
Also, for any fixed $T>0$, 
we have $v(x,t) \in [0,C_1(T)]$ 
a.e.\ in $(x,t) \in \overline{\Omega} \times [0,T]$ by  \eqref{vepbounds}.
As for (a) in Definition \ref{Def:WS}, we infer from \eqref{shuusoku2} 
that $u \in C_{\mathrm{w}}([0,\infty);L^2(\Omega))$.
We consider (b) in Definition \ref{Def:WS}.
We observe that $v \in C([0,\infty);L^2(\Omega))$ in \eqref{shuusoku2}, 
which along with
\eqref{vepesti} leads to $v \in L^{2}_{\mathrm{loc}}([0,\infty);H^2(\Omega))$.
Next, we confirm (c) in Definition \ref{Def:WS}. 
Due to \eqref{shuusoku1}, we see that $D^{\ep_k}(u^{\ep_k}, v^{\ep_k}) \to D(u, v)$ 
a.e.\ in $\Omega \times (0, T)$ as $k \to \infty$.
By \eqref{hanaDep}, \eqref{clabounds}, \eqref{Dep} and the fact that 
$\ep_k<1$, we have
\begin{equation*}
\mathcal{D}^{\ep_k}(u^{\ep_k},v^{\ep_k})
\le
\int_0^1D^{\ep_k}(r,v^{\ep_k}) \, dr
\le
\sup_{\substack{0\le r\le1 \\ 0\le s\le C_1(T)}}D(r,s)+1
\le C_7(T)
\end{equation*}
for some $C_7(T)>0$ independent of $k$.
Hence, the dominated convergence theorem yields
\begin{equation*} 
\mathcal{D}^{\ep_k}(u^{\ep_k}, v^{\ep_k}) \to \mathcal{D}(u, v) \quad 
\mbox{strongly in} \ L^2(0, T; L^2(\Omega)) 
\end{equation*}
as $k \to \infty$.
Hence, recalling that 
$\{\nabla[\mathcal{D}^{\ep_k}(u^{\ep_k}, v^{\ep_k})]\}_{k \in \N}$ 
is bounded in $L^2(0,T;(L^2(\Omega))^N)$ by \eqref{Con:con}, 
we see that $\mathcal{D}(u, v) \in L^2(0, T;H^1(\Omega))$ and
\begin{equation}\label{1}
\nabla[\mathcal{D}^{\ep_k}(u^{\ep_k}, v^{\ep_k})] \to \nabla[\mathcal{D}(u, v)] 
\quad \mbox{weakly in}\  L^2(0,T;(L^2(\Omega))^N)
\end{equation}
as $k \to \infty$.
Next, since
$(\mathcal{D}^{\ep_k})_s(u^{\ep_k},v^{\ep_k}) \to \mathcal{D}_s(u,v) \ 
\mbox{a.e. in} \ \Omega \times (0,T)$ as $k \to \infty$,
we deduce from \eqref{Con:hanaDeppas} that 
\begin{equation*} 
(\mathcal{D}^{\ep_k})_s(u^{\ep_k},v^{\ep_k}) \to \mathcal{D}_s(u,v) \quad 
\mbox{weakly$^*$ in} \ L^\infty(0,T; L^\infty(\Omega))
\end{equation*}
as $k \to \infty$.
According to \eqref{vepesti}, we have 
\begin{equation}\label{finally}
\nabla v^{\ep_k} \to \nabla v \quad \mbox{weakly in}\ L^2(0,T;(L^2(\Omega))^N)
\end{equation}
as $k \to \infty$.
Thus we obtain 
\begin{equation} \label{@}
(\mathcal{D}^{\ep_k})_s(u^{\ep_k},v^{\ep_k})\nabla v^{\ep_k} \to \mathcal{D}_s(u,v) \nabla v 
\quad \mbox{weakly in}\ L^2(0,T;(L^2(\Omega))^N)
\end{equation}
as $k \to \infty$.
Setting 
\[
D(u,v)\nabla u
 :=\nabla[\mathcal{D}(u,v)]-\mathcal{D}_s(u,v)\nabla v,
 \]
 we have 
 $D(u,v)\nabla u \in  L^2(0,T;(L^2(\Omega))^N)$ 
by relying on 
the facts $\mathcal{D}(u,v)\in L^2(0,T;H^1(\Omega))$ 
and $ \mathcal{D}_s(u,v) \nabla v \in L^2(0,T;(L^2(\Omega))^N)$.
Thus (c) in Definition \ref{Def:WS} is verified. 
Finally, we check (d) in Definition \ref{Def:WS}.
A combination of \eqref{1} and \eqref{@} yields
\begin{equation} \label{<}
D^{\ep_k} (u^{\ep_k}, v^{\ep_k})\red{\nabla u}^{\ep_k} \to D(u,v)
\red{\nabla u} \quad 
\mbox{weakly in}\ L^2(0,T;(L^2(\Omega))^N)
\end{equation}
as $k \to \infty$.
On the other hand, we know that 
\begin{equation} \label{kokodemo}
h(u^{\ep_k},v^{\ep_k}) \nabla v^{\ep_k} \to h(u,v) \nabla v \quad 
\mbox{weakly in}\ L^2(0,T;(L^2(\Omega))^N)
\end{equation}
as $k \to \infty$.
Indeed, we have
$h(u^{\ep_k},v^{\ep_k}) \to h(u,v)$ a.e.\ in $\Omega \times (0,T)$ 
as $k \to \infty$ by \eqref{shuusoku1}.
Also, we see from \eqref{clabounds} and \eqref{shuusoku1} that 
$h(u^{\ep_k},v^{\ep_k}) \to h(u,v)$ weakly$^*$ in $ L^\infty(0,T; L^\infty(\Omega))$
as $k \to \infty$.
Therefore, we obtain \eqref{kokodemo} by \eqref{finally}.
From \eqref{shuusoku1} it follows that 
\begin{equation}\label{O}
u^{\ep_k} \to u \quad \mbox{weakly in}\ L^2(0,T;L^2(\Omega))
\end{equation}
as $k \to \infty$.
Let $\varphi \in H^1(0,T;H^1(\Omega))$ with $\varphi(T) = 0$.
Multiplying the first equation in \eqref{Sys:Mainkinji} by $\varphi$ 
and integrating it over $\Omega$, we obtain 
\begin{equation*} 
\frac{d}{dt}\int_\Omega u^{\ep_k} \varphi \, dx- \int_\Omega u^{\ep_k} \varphi_t \, dx
=-\int_\Omega  (D^{\ep_k}(u^{\ep_k},v^{\ep_k}) \nabla u^{\ep_k} 
                         - h(u^{\ep_k},v^{\ep_k}) \nabla v^{\ep_k}) \cdot \nabla \varphi \, dx
\end{equation*}
by the divergence theorem.
Integrating this identity over $(0,T)$, we observe that
\begin{equation*} 
-\!\int_0^T\!\! \int_\Omega u^{\ep_k} \varphi_t \, dxdt 
+\int_0^T \!\!\int_\Omega  (D^{\ep_k}(u^{\ep_k},v^{\ep_k}) \nabla u^{\ep_k} 
                           - h(u^{\ep_k},v^{\ep_k}) \nabla v^{\ep_k}) \cdot \nabla \varphi \, dxdt 
=\int_\Omega u^{\ep_k}_0 \varphi(0) \, dx
\end{equation*}
since $\varphi(T)=0$. 
Passing to the limit $k \to \infty$, we obtain
\begin{equation*}
-\int_0^T \int_\Omega u \varphi_t \, dxdt
+\int_0^T \int_\Omega  (D(u,v) \nabla u 
                                    - h(u,v) \nabla v) \cdot \nabla \varphi \, dxdt
=\int_\Omega u_0 \varphi(0) \, dx
\end{equation*}
from \eqref{<}, \eqref{kokodemo} and \eqref{O} together with \eqref{inikinji}.
Similarly, we can see that
\[
 -\int_0^T \int_\Omega v \varphi_t \, dxdt +
 \int_0^T \int_\Omega \nabla v \cdot \nabla \varphi \, dxdt 
 - \int_0^T \int_\Omega 
 \red{g(u,v) \varphi} \, dxdt
 = \int_\Omega v_0 \varphi(0) \, dx.
\]
Finally, based on \eqref{shituryouhozonn2},
we can obtain \eqref{shituryouhozonn}.
This completes the proof of Theorem \ref{Thm:exist}.
\end{proof}
\section{Uniqueness of global weak solutions} \label{Sec:GWSU}
In this section we suppose that $D(r,s) = D(r)$
for all $r\in[0,1]$ and $s\in[0,\infty)$, and
in addition to \eqref{Con:D}, \eqref{Con:h}, \eqref{Con:g} 
and \eqref{Con:ini}, assume further
\eqref{SCon:Dh} and \eqref{SCon:g} and that $v_0 \in W^{2,p}(\Omega)$ 
with $p >N$.
\begin{proof}[Proof of Theorem \ref{Thm:unique} (Uniqueness)]
We let $\mathcal{N}w$ denote the unique solution 
$\varphi\in H^2(\Omega)$ of the problem 
\begin{equation}\label{Def:N}
\begin{cases}
  -\Delta\varphi = w,  &x\in\Omega, \\
  \nabla\varphi \cdot \nu = 0,   &x\in \partial \Omega,
\end{cases}
\end{equation}
with
\[
    \int_\Omega \varphi(x)\,dx = 0
\]
for $w \in L^2(\Omega)$ with $\int_\Omega w(x)\,dx = 0$.
Let $(u,v)$ and $(\widehat{u},\widehat{v})$ be global weak solutions 
of \eqref{Sys:Main} satisfying \eqref{shituryouhozonn} in the sense 
of Definition \ref{Def:WS}. Set
\[
U := u-\widehat{u} \quad \mbox{and} \quad
V := v-\widehat{v}.
\]
Let $T>0$ and $t \in [0,T]$.
Since $(u,v)$ and $(\widehat{u},\widehat{v})$ are global weak solutions, 
we see from
\cite[p.108, Proposition 2.1 (b)$\Rightarrow$(a)]{Showalter-1997} 
that for all $\psi \in L^2(0,t;H^1(\Omega))$,
\begin{align*}
&\int_0^t \red{\langle U_t, \varphi \rangle
_{(H^1(\Omega))', H^1(\Omega)} \, ds} \\
\label{eq:SW}
&= -\int_0^t \int_\Omega 
(D(u)\nabla u - D(\widehat{u})\nabla \widehat{u} 
-(h(u,v)\nabla v-h(\widehat{u},\widehat{v})\nabla v + h(\widehat{u},\widehat{v})\nabla V))
\cdot \nabla \psi \, dxds.
\end{align*}
Since it is assumed that $u$ and $\hat{u}$ satisfy
\eqref{shituryouhozonn}, 
we have $\int_\Omega U \,dx=0$.
Taking $\psi = \mathcal{N}U \in L^2(0,t;H^1(\Omega))$ 
in the above identity and noting that
\[
\red{\langle U_t, \mathcal{N}U \rangle
_{(H^1(\Omega))', H^1(\Omega)}} 
= \frac12\cdot \frac{d}{dt} \int_\Omega | \nabla(\mathcal{N}U)(t) |^2 \, dx
\]
and that $\int_\Omega |\nabla (\mathcal{N}U)(0)|^2 \, dx = 0$, we have 
\begin{align}
\nonumber
\frac12 \int_\Omega | \nabla(\mathcal{N}U)(t) |^2 \, dx 
&= -\int_0^t \int_\Omega 
(D(u)\nabla u - D(\widehat{u})\nabla \widehat{u}) \cdot \nabla (\mathcal{N}U)\, dxds \\
\nonumber
&\quad \, +\int_0^t \int_\Omega(h(u,v)\nabla u-h(\widehat{u},\widehat{v})\nabla v 
                                                + h(\widehat{u},\widehat{v})\nabla V)
\cdot \nabla (\mathcal{N}U) \, dxds \\[2mm]
\label{Def:I(t)}
&=: I(t).
\end{align}
Since $u$ and $v$ are bounded in $\Omega \times (0,t)$ 
and it is assumed that $v_0 \in W^{2,p}(\Omega)$ with $p >N$ and $\nabla v_0 \cdot \nu =0$ on $\partial \Omega$, 
the maximal Sobolev regularity 
for parabolic equations \cite[3.1 Theorem]{HP-1997-CPDE} 
implies that $v \in L^p(0,t;W^{2,p}(\Omega))$ for $p > N$.
Thus
we obtain $\nabla v \in L^2(0,T;(L^\infty(\Omega))^N)$
by the Sobolev embedding.
From the Green formula, the definition of $\mathcal{N}U$ (see \eqref{Def:N})
and the Schwarz inequality we obtain
\begin{align*}
I(t)
&= \int_0^t \int_\Omega (\mathcal{D}(u)-\mathcal{D}(\widehat{u})) 
\Delta(\mathcal{N}U)
\, dxds \\
&\quad \, + \int_0^t \int_\Omega (h(u,v)-h(\widehat{u},\widehat{v}))\nabla v
\cdot \nabla(\mathcal{N}U) \, dxds \\
&\quad \, + \int_0^t \int_\Omega h(\widehat{u},\widehat{v}) \nabla V
\cdot \nabla(\mathcal{N}U) \, dxds \\
&\le - \int_0^t \int_\Omega (\mathcal{D}(u)-\mathcal{D}(\widehat{u}))U \, dxds \\
&\quad \, + \int_0^t \| \nabla v\|_{L^\infty(\Omega)}
\| \nabla(\mathcal{N}U) \|_{L^2(\Omega)}
\| h(u,v)-h(\widehat{u},\widehat{v}) \|_{L^2(\Omega)} \, ds \\
&\quad \, + \| h \|_{L^\infty((0,1) \times (0,\red{K(T)}))}
\int_0^t \| \nabla(\mathcal{N}U) \|_{L^2(\Omega)}
\| \nabla V \|_{L^2(\Omega)} \, ds,
\end{align*}
where $\red{K(T)} := \max \{ \| v \|_{L^\infty(\Omega \times (0,T))},
\| \widehat{v} \|_{L^\infty(\Omega \times (0,T))} \}$.
Let $\delta > 0$.
Applying the Young inequality to the second and third terms 
on the right-hand side derives
\begin{align*}
I(t) &\le - \int_0^t \int_\Omega(\mathcal{D}(u)-\mathcal{D}(\widehat{u}))U \, dxds 
\\
&\quad \, + \frac{\delta}{2} \int_0^t \int_\Omega (h(u,v)-h(\widehat{u},\widehat{v}))^2 \, dxds
+ \frac{1}{2\delta} 
\int_0^t \| \nabla v\|_{L^\infty(\Omega)}^2
\| \nabla(\mathcal{N}U) \|_{L^2(\Omega)}^2 \, ds \\
&\quad \, + \frac{\delta}{2} \int_0^t \| \nabla V \|_{L^2(\Omega)}^2 \, ds
+ \frac{1}{2\delta} \| h \|_{L^\infty((0,1) \times \red{(0,K(T)))}^2}
\int_0^t \| \nabla(\mathcal{N}U) \|_{L^2(\Omega)}^2 \, ds .
\end{align*}
In light of the condition \eqref{SCon:Dh}, we obtain
\begin{align*}
I(t) &\le - \int_0^t \int_\Omega (\mathcal{D}(u)-\mathcal{D}(\widehat{u}))U \, dxds \\
&\quad \, 
+ \frac{\delta}{2} \red{C_0(T)} 
   \int_0^t \int_\Omega (\mathcal{D}(u)-\mathcal{D}(\widehat{u}))U \, dxds
 + \frac{\delta}{2} \red{C_1(T)} \int_0^t \int_\Omega V^2 \, dxds \\
&\quad \, 
+ \frac{1}{2\delta} 
\int_0^t \| \nabla v\|_{L^\infty(\Omega)}^2
\| \nabla(\mathcal{N}U) \|_{L^2(\Omega)}^2 \, ds \\
&\quad \,
+ \frac{\delta}{2} \int_0^t \| \nabla V \|_{L^2(\Omega)}^2 \, ds
+ \frac{1}{2\delta} \| h \|_{L^\infty((0,1) \times (0,\red{K(T)}))}^2
\int_0^t \| \nabla(\mathcal{N}U) \|_{L^2(\Omega)}^2 \, ds 
\end{align*}
\red{
for some positive constants 
$C_0(T)$ and $C_1(T)$.}
Fixing $\delta$ as $0<\delta<\min\big\{\frac{2}{C_0},1\big\}$, 
we have
\begin{align}
\nonumber
I(t) &\le
\frac{\delta}{2} \red{C_1(T)} \int_0^t \int_\Omega V^2 \, dxds
+ \frac{1}{2\delta} \int_0^t 
\| \nabla v\|_{L^\infty(\Omega)}^2
\| \nabla(\mathcal{N}U) \|_{L^2(\Omega)}^2 \, ds \\
\label{ine:I(t)}
&\quad \,
+ \frac{\delta}{2}
\int_0^t \| \nabla V \|_{L^2(\Omega)}^2 \, ds
+ \frac{1}{2\delta}
\| h \|_{L^\infty((0,1) \times (0,\red{K(T)}))}^2
\int_0^t \| \nabla(\mathcal{N}U) \|_{L^2(\Omega)}^2 \, ds.
\end{align}
From \eqref{Def:I(t)} and \eqref{ine:I(t)} it follows that 
\begin{align}
\nonumber
\| \nabla (\mathcal{N}U)(t) \|_{L^2(\Omega)}^2 
&\le 
\delta \red{C_1(T)} \int_0^t \int_\Omega V^2 \, dxds
+ \frac{1}{\delta}\int_0^t \| \nabla v\|_{L^\infty(\Omega)}^2
\| \nabla(\mathcal{N}U) \|_{L^2(\Omega)}^2 \, ds \\
\label{ine:nabNU}
&\quad \, 
+ \delta
\int_0^t \| \nabla V \|_{L^2(\Omega)}^2 \, ds
+ \frac{1}{\delta}
\| h \|_{L^\infty((0,1) \times (0,\red{K(T)}))}^2
\int_0^t \| \nabla(\mathcal{N}U) \|_{L^2(\Omega)}^2 \, ds.
\end{align}
Similarly, we see from (d) in Definition \ref{Def:WS}, 
\cite[p.108, Proposition 2.1 (b)$\Rightarrow$(a)]{Showalter-1997} and
\eqref{SCon:g} that
\begin{align}
\nonumber
&\| V(t) \|_{L^2(\Omega)}^2
+ \int_0^t \| \nabla V \|_{L^2(\Omega)}^2 \, ds \\
\label{ine:FS}
&\le C(T) \int_0^t \| V \|_{L^2(\Omega)}^2 \, ds
+ C(T) \int_0^t \big(1+\| \nabla v \|_{L^\infty(\Omega)}^2\big)
\| \nabla(\mathcal{N}U) \|_{L^2(\Omega)}^2 \, ds,
\end{align}
where $C(T)$ is a positive constant depending on $T$.
Thanks to \eqref{ine:nabNU} and \eqref{ine:FS}, we have
\begin{align*} 
\nonumber
&\| V(t) \|_{L^2(\Omega)}^2 + \| \nabla(\mathcal{N}U)(t) \|_{L^2(\Omega)}^2 
\\
\label{ine:Gron}
&\le C(T) \int_0^t \big(2+\| \nabla v \|_{L^\infty(\Omega)}^2\big)
\big(\| V \|_{L^2(\Omega)}^2 +
\| \nabla(\mathcal{N}U) \|_{L^2(\Omega)}^2\big) \, ds.
\end{align*}
Recalling that $\nabla v \in L^2(0, T; (L^\infty(\Omega))^N)$ as noted above, 
we can use the Gronwall lemma to see that
$ V(t) = \nabla(\mathcal{N}U)(t) = 0$ for every $t \in [0,T]$.
Therefore, from \eqref{Def:N} and the Green formula we conclude that
\[
V(t) =U(t) = 0
\]
for all $t \in [0,T]$,
which completes the proof.
\end{proof}
\section{Stationary solutions} \label{Sec:SS}
In this section we construct a flat-hump-shaped stationary solution 
of \eqref{Sys:Main} in the one-dimensional setting.  
\subsection{Basic assumption, definition and properties}
Throughout this section, 
\red{we restrict ourselves 
to the case $\Omega = (0, \ell)$ 
$(\ell > 0)$}
and 
we put 
\begin{equation} \label{Eq:g}
g(r,s) := \gamma r - \beta s
\end{equation}
for some $\beta > 0$ and $\gamma > 0$, and assume that 
\[
D(r,s) = D_1(r) D_2(s) \quad
\mbox{and}
\quad h(r,s) = h_1(r) h_2(s)
\]
for all $r\in[0,1]$ and $s\in[0,\infty)$,
where
\begin{align} \label{Con:Dh1}
& \begin{cases}
 D_1, h_1 \in C^2([0,1]), \\
 D_1(1) = 0, \quad D_1(r) > 0 \quad (r \in [0,1)), \\
 h_1(0) = h_1(1) = 0, \quad h_1(r) > 0 \quad (r \in (0,1)),
 \end{cases}
\\[2mm]
\label{Con:Dh2}
 & \begin{cases}
 D_2, h_2 \in C^2([0,\infty)), \\
 D_2(s) > 0 \quad (s \in [0,\infty)), \\
 h_2(s) > 0 \quad (s \in [0,\infty)).
 \end{cases}
\end{align}
\begin{df}[Stationary solutions] \label{Df:SS} 
A couple $(u,v)$ will be called 
a \textit{stationary solution} of \eqref{Sys:Main} if 
 \begin{enumerate}[(i)]
 \item $u,v \in L^\infty(\red{0,\ell}), \ \, v \in H^2(\red{0,\ell}), \ \, \mathcal{D}_1(u) 
 \in H^1(\red{0,\ell})$, 
 where $\mathcal{D}_1(r) := \int_0^r D_1(\sigma) \, d\sigma$,
 \item a couple $(u,v)$ satisfies
 \begin{empheq}[left={\empheqlbrace}]{alignat=2}
\label{Con:uvbounds}
  & (u(x),v(x)) \in [0,1] \times [0,\infty) &\qquad 
  &\mbox{a.e.\ in} \; \red{(0,\ell)}, \\
\label{Eq:1st}
  & D_2(v) \red{(\mathcal{D}_1(u))'} - h_1(u) h_2(v) \red{v'} = 0 & 
  &\mbox{a.e.\  in} \; \red{(0,\ell)}, \\
  \label{Eq:2nd}
  & \red{{}-v''} + \beta v = \gamma u & 
  &\mbox{a.e.\ in} \; \red{(0,\ell)}, \\
  \label{Eq:B}
  & \red{v'(0) = v'(\ell) = 0.} & 
  & 
\end{empheq}
 \end{enumerate}
If the couple $(u,v)$ additionally satisfies
\[
\red{\int_{0}^{\ell} u(x) \, dx = \frac{\beta}{\gamma} \int_{0}^{\ell} v(x) \, dx = M}
\]
for \red{$M \in [0,\ell]$}, 
then the couple $(u,v)$ will be called
a \textit{stationary solution of \eqref{Sys:Main} 
with mass $(M,\frac{\gamma M}{\beta})$}. 
\red{In addition, 
a stationary solution $(u,v)$ 
of (1.1) will be said 
to be {\it flat hamp-shaped} 
if there exist 
$x_1 \in (0,\frac{\ell}{2})$ 
and 
a strictly increasing function 
$w$ on $[0,x_1)$ 
such that 
\begin{align*}
u(x) =
\begin{cases}
w(x) \in [0,1)  &\mbox{for} \ x \in [0, x_1), \\
1                 &\mbox{for} \ x \in [x_1,\ell-x_1], \\
w(x_1-x) \in [0,1)  
&\mbox{for} \ x \in (\ell-x_1, \ell].
\end{cases}
\end{align*}}%
\end{df}
Let us define 
\begin{align*}
j(r,s) &:= \int_{\frac12}^{r} \frac{D_1(\sigma)}{h_1(\sigma)} \, d\sigma
- \int_0^s \frac{h_2(\sigma)}{D_2(\sigma)} \, d\sigma, \\
j_1(r) &:= \int_{\frac12}^{r} \frac{D_1(\sigma)}{h_1(\sigma)} \, d\sigma, \\
j_2(s) &:= \int_0^s \frac{h_2(\sigma)}{D_2(\sigma)} \, d\sigma.
\end{align*}
for $(r,s)\in [0,1]\times [0,\infty)$.
Then we give the basic properties for $j_1$, $j_2$ and $j$.
\begin{lem} \label{L}
The function $j_1$ is a strictly increasing function from $(0,1)$ onto 
$(-\infty,j_1(1))$.
Moreover we have $\frac{D_1}{h_1} 
\notin L^1\big(0,\frac12\big)$.
\red{Also, the function $j_2$ 
is a strictly increasing function
on $(0,\infty)$.}
\end{lem}
\begin{proof}\label{P1}
It is clear that $j_1$ and $j_2$ \red{are} a strictly increasing 
by \eqref{Con:Dh1} and \eqref{Con:Dh2}.
Moreover, by the mean value theorem, we see that 
\[
\frac{D_1(\sigma)}{h_1(\sigma)} 
\ge \frac{\min_{[0,\frac12]}D_1}{\| h_1' \|_{L^\infty(0,1)}}
\cdot \frac{1}{\sigma}
\]
for $\sigma \in \big(0,\frac12\big)$, and hence $\frac{D_1}{h_1} 
\notin L^1\big(0,\frac12\big)$.
\end{proof}
\begin{prop} \label{P2}
Let $M \in (0, \red{\ell})$ and let $(u,v)$ be a stationary solution 
of \eqref{Sys:Main} with mass $(M,\frac{\gamma M}{\beta})$.
Set 
\red{$(0,\ell)_u := \{ x \in (0,\ell) \,| \, u(x) \in [0,1) \}$.} 
Then $\red{(0,\ell)_u}$ is an open subset 
of $\red{(0,\ell)}$ 
and $\red{(0,\ell)_u} \neq \emptyset$. 
Also, the function $j(u,v)$ is constant 
on each connected component of $\red{(0,\ell)_u}$ 
and $u \in C(\red{[0,\ell]}) 
\cap C^1(\red{(0,\ell)_u})$.
Moreover, $u(x) > 0$ for all $x \in \red{[0,\ell]}$ 
and 
\begin{equation} \label{Con:vbound}
v(x) \in \left[0,\frac{\gamma}{\beta}\right] \quad 
\forall\, x \in \red{[0,\ell]}.
\end{equation}
\end{prop}
\begin{proof}
From \eqref{Eq:2nd}, \eqref{Eq:B} 
and the fact $u \in L^\infty(\red{0,\ell})$ along with
the standard regularity result for elliptic equations, we have 
$v \in W^{2,p}(\red{0,\ell})$ for all $p \in (1,\infty)$.
Taking $p$ such that $p>N$, we obtain $\red{v'} \in L^\infty(\red{0,\ell})$
by the Sobolev embedding.
Since $D_2(v) > 0$, it follows from \red{\eqref{Con:Dh2}} and \eqref{Eq:1st} that
\red{$(\mathcal{D}_1(u))'
= \frac{1}{D_2(v)}h_1(u) h_2(v) v' \in L^\infty(0,\ell)$.}
Thus 
$\mathcal{D}_1(u) \in W^{1,\infty}(\red{0,\ell})$.
Using the Rellich--Kondrachov theorem, 
we have $\mathcal{D}_1(u) \in C(\red{[0,\ell]})$, 
which together with the fact that $\mathcal{D}_1$ is increasing
yields $u \in C(\red{[0,\ell]})$.
Taking $p$ large enough, we obtain $v\in C^1(\red{[0,\ell]})$.
Again by the identity
\red{$(\mathcal{D}_1(u))'
= \frac{1}{D_2(v)}h_1(u) h_2(v) v'$,} we have 
$\mathcal{D}_1(u) \in C^1(\red{[0,\ell]})$.
By the assumptions that 
$u \ge 0$ in $\red{(0,\ell)}$ 
and that $\red{M<\ell}$, we make sure that 
$\red{(0,\ell)}_u$ is an open subset of $\red{(0,\ell)}$ 
and $\red{(0,\ell)}_u \neq \emptyset$.
Since $\mathcal{D}_1$ is strictly increasing on $\red{(0,\ell)}_u$, 
we observe that $u \in C^1(\red{(0,\ell)}_u)$.

\red{
We now show
the positivity of $u$ 
throughout $[0,\ell]$. 
In order to achieve this,  
we aim at verifying 
the positivity of $u$ 
on $[\ep,\ell-\ep]$ 
for $\ep > 0$ sufficiently small. 
To this end,
let $\ep > 0$ 
sufficiently small,  
let $U_a$ $(a \in A)$ 
be a connected component of 
$\{ x \in (0,\ell) 
\,|\, u(x) \in (0,1) \}$ 
and 
let
$V_b$ $(b \in B)$ 
be 
a connected component of 
$\{ x \in (0,\ell) 
\,|\, u(x) \in \{0,1\} \}$.  
Then since 
$[\ep, \ell-\ep]$ is compact 
and 
$U_a$ is open for all $a \in A$
(by virtue of the fact that 
the set 
$\{ x \in (0,\ell) 
\,|\, u(x) \in (0,1) \}$ 
is open) 
as well as 
$V_b + B(0, \delta_b)$ is open 
for all $b \in B$, 
there exist 
$a_1, a_2, \ldots, a_n \in A$ 
and 
$b_1, b_2, \ldots, b_m \in B$ 
such that 
\begin{align*} 
[\ep, \ell-\ep] 
\subset 
\Big( \bigcup_{i=1}^{n} U_{a_i} \Big)
\cup 
\Big( 
\bigcup_{j=1}^{m} 
\big( V_{b_j} + B(0;\delta_{b_j}) \big) 
\Big), 
\end{align*}
where 
$\delta_{b_j} > 0$ 
is a constant 
to be determined later. 
We here fix $x \in [\ep, \ell-\ep]$. 

Consider the case that 
$x \in U_{a_i}$ 
for some 
$i \in \{ 1, 2, \ldots, n \}$. 
Then,  
noting the fact that $0 < u(x) < 1$  
so that $h_1(u(x)) \ne 0$, 
we see that 
\begin{align*}
\frac{D_1(u(x))}{h_1(u(x))}u'(x)
-\frac{h_2(v(x))}{D_2(v(x))}v'(x)
=0,
\end{align*}
which yields 
$(j(u,v))'(x)=0$. 
Therefore, 
there exists $\lambda_i \in \R$ 
such that 
\begin{align} \label{bibunnzero}
j_1(u(x))=j_2(v(x))+\lambda_i.
\end{align} 
Thanks to 
\eqref{bibunnzero} 
and 
Lemma \ref{L}, 
we observe that 
\begin{align} \label{unoseitisei1}
u(x)
=j_1^{-1}(j_2(v(x))+\lambda_i) 
>j_1^{-1}(\lambda_i) > 0. 
\end{align}

Consider the case that  
$x \in V_{b_j} + B(0; \delta_{b_j})$ 
for some  
$j \in \{ 1, 2, \ldots, m \}$. 

\noindent
(i) If $x \in  V_{b_j}$, 
we infer from 
the definition of $V_{b_j}$ 
that 
$u(x) \in \{ 0, 1 \}$. 
Since 
$u$ is continuous 
and 
$(u,v)$ is a stationary solution 
of \eqref{Sys:Main} 
with mass 
$(M, \frac{\gamma M}{\beta})$, 
we conclude that $u(x) = 1$. 

\noindent
(ii) If $x \notin  V_{b_j}$, 
we choose $\delta_{b_j} > 0$ 
such that 
$u(y) \ge \frac12$ 
for all  
$y \in V_{b_j} + B(0; \delta_{b_j})$, 
which is guaranteed 
by the continuity of $u$.
Thus we have 
$u(x) \ge \frac12.$
Summarizing (i) and (ii), 
we obtain that  
\begin{align*} 
u(x) \ge \frac12 
\quad
\forall \, 
x \in V_{b_j} + B(0; \delta_{b_j}), 
\end{align*}
which 
together with 
\eqref{unoseitisei1} 
entails that 
\begin{align} \label{unoseitisei}
u(x) 
\ge 
\min \Big\{ \frac12, 
            j_{1}^{-1}(\lambda_1), 
            j_{1}^{-1}(\lambda_2),  
           \ldots, 
            j_{1}^{-1}(\lambda_n) 
                                    \Big\}
                              > 0
\quad 
\forall \, x \in [\ep, \ell-\ep].         
\end{align}
As $\ep > 0$ is arbitrary, 
the estimate \eqref{unoseitisei} 
implies 
the positivity of $u$ 
in $(0,\ell)$, 
which along with 
the continuity of $u$ 
asserts that 
\begin{align*} 
u(x) > 0
\quad
\forall \, 
x \in [0,\ell]. 
\end{align*}
This also shows that 
$(0,\ell)_u 
= \{ x \in (0,\ell) \,| 
\, u(x) \in (0,1) \}$.}

Finally, the comparison principle for 
\eqref{Eq:2nd} and \eqref{Eq:B} leads to \eqref{Con:vbound}.
\end{proof}

In the case that $\frac{D_1}{h_1} \notin L^1(\frac{1}{2},1)$, 
we can give lower and upper bounds for $u$, which generalizes 
\cite[Proposition 8]{LW-2005-PNDEA}.
\begin{prop}\label{*}
Let $M\in (0,\red{\ell})$ and let the couple $(u,v)$ be a stationary solution 
of \eqref{Sys:Main} with mass $(M,\frac{\gamma M}{\beta})$.
Assume that $j_1(1)= \infty$. Then we have 
\[
0<j_1^{-1}\left(j_1\left(\frac{M}{\red{\ell}}\right)
                     -j_2\left(\frac{\gamma}{\beta}\right)\right)\le u(x)
\le j_1^{-1}\left(j_1\left(\frac{M}{\red{\ell}}\right)
                       +j_2\left(\frac{\gamma}{\beta}\right)\right)<1
\]
for all $x\in \red{[0,\ell]}$.
\end{prop}
\begin{proof}
Let \red{$U_a$ $(a \in A)$} 
be a set as defined 
in the proof Proposition \ref{P2}. 
Then 
there exists 
\red{$\lambda_a \in \R$} 
such that
\begin{equation} \label{Eq}
j_1(u(x))=j_2(v(x))+\red{\lambda_a} 
\end{equation}
for all \red{$x \in U_a$}.
In view of Lemma \ref{L}, the property \eqref{Con:vbound} 
and the assumption $j_1(1)=\infty$, we obtain 
\[
u(x)\le j_1^{-1}\left(j_2\left(\frac{\gamma}{\beta}\right)+\lambda\right) <1
\]
for all $\red{x \in U_a}$. 
Since $u \in C(\red{[0,\ell]})$ by Proposition \ref{P2}, 
this implies that 
\red{$U_a=(0,\ell)_u=(0,\ell)$.} 
Thus by assumption, we see from \eqref{Con:vbound} and \eqref{Eq} that
\[
M=\red{\int_{0}^{\ell} u(x) \, dx 
=\int_{0}^{\ell}  j_1^{-1}\left(j_2\left(v\right)+\lambda\right) \, dx}
\in \left[\red{\ell}j_1^{-1}(\lambda), 
             \red{\ell}j_1^{-1}\left(j_2\left(\frac{\gamma}{\beta}\right)
            +\lambda\right)\right],
\]
from which we have
\[
\lambda \in \left[ j_1\left(\frac{M}{\red{\ell}}\right) 
                - j_2\left(\frac{\gamma}{\beta}\right), 
                         j_1\left(\frac{M}{\red{\ell}}\right)\right],
\]
which along with Lemma \ref{L} 
and \eqref{Eq} leads to the conclusion.
\end{proof}

\subsection{Flat-hump-shaped stationary solutions} \label{Sec:SS2}
\red{In what follows, we focus on the case
$\frac{D_1}{h_1} \in L^1\big(\frac{1}{2},1\big)$.}
As a preparation, we introduce the function 
\[
f_{\lambda}(s):=j_1^{-1}(j_2(s)+\lambda), \quad \red{0 \le s \le v_\lambda
:=j_2^{-1}(j_1(1)-\lambda)}
\]
for $\lambda\in \R$, and let
\begin{equation*}
\overline{f_\lambda}(s):=
\begin{cases}
f_\lambda(s) \quad &\mbox{for} \ \red{s \in [0, v_\lambda]}, \\
1 &\mbox{for} \ \red{s \in [v_\lambda, \infty)}
\end{cases}
\end{equation*}
for $\lambda\in \R$.
In order to state our result we need the following condition 
on the triplet $(\lambda, \gamma, \beta) \in \R \times (0,\infty)^2$:
\begin{equation}
\exists \, \rho_{-1}>0 \ \exists \, \rho_0>0
\quad \mbox{s.t.\quad}
\\[3mm]
\begin{cases}\label{C}
\rho_{-1}<\rho_0<v_\lambda<\dfrac{\gamma}{\beta},&  \\
\overline{f_\lambda}(\rho_i)=\dfrac{\beta}{\gamma}\rho_i \quad 
\mbox{for} \ i \in \{-1,0\}, & \\[3mm]
\overline{f_\lambda}(s)>\dfrac{\beta}{\gamma}s
\quad \ \mbox{for} \ s \in \left(\rho_0,\dfrac{\gamma}{\beta}\right),&  \\[3mm]
\overline{f_\lambda}(s)<\displaystyle\frac{\beta}{\gamma}s
\quad \  \mbox{for} \ s \in \displaystyle(\rho_{-1},\rho_0).& 
\end{cases}
\end{equation}
\begin{thm}[Existence of flat-hump-shaped stationary solutions] \label{SSE}
Suppose that the triplet $(\lambda,\gamma, \beta)\in \R \times (0,\infty)^2$ 
fulfills the condition \eqref{C}.
Assume further that
\begin{equation} \label{jyoukenn}
\frac{1}{v_\lambda-\rho_{-1}}\int_{\rho_{-1}}^{v_\lambda} f_\lambda(s) \, ds 
< \frac{\beta}{\gamma} \cdot \red{\frac{v_\lambda+\rho_{-1}}{2}}.
\end{equation}
Then for $\red{\ell}>0$ large enough, 
there is a flat-hump-shaped stationary solution 
$(u,v)$ of \eqref{Sys:Main}.
Moreover, there exists $x_1\in (0,\frac{l}{2})$ such that
\[
\begin{cases}
u(x) \in [0,1)  &\mbox{for} \ x \in [0, x_1)\cup (\red{\ell}-x_1,\red{\ell}], \\
u(x)=1 &\mbox{for} \ x \in [x_1,\red{\ell}-x_1], \\
v(x)\ge v_\lambda &\mbox{for} \ x \in [x_1, \red{\ell}-x_1].
\end{cases}
\]
In addition, if 
\begin{equation}\label{j1-1'}
(j_1^{-1})' \in L^2(j_1(1)-\delta, j_1(1)) \quad \mbox{for some} \ \delta >0,
\end{equation}
then we have $u\in H^1(0,\red{\ell})$.
\end{thm}
\begin{proof}
The proof is divided into three steps.

\medskip
\noindent
{\bf \boldmath Step 1. Preparations.}
We consider the one-dimensional boundary-value problem 
\[
\begin{cases}
v''=\widetilde{g}(v) \quad \mbox{in} \ (0,\red{\ell}), \\
\red{v'(0)=v'(\red{\ell})=0,}
\end{cases}
\]
where
\begin{equation}\label{tuikanomono}
\widetilde{g}(v):=-\gamma \overline{f_\lambda}(v)+\beta v.
\end{equation}
Note that this problem is equivalent to the hamiltonian system
\begin{equation}\label{ham}
\begin{cases}
v'=w, \\
w'=\widetilde{g}(v),
\end{cases}
\end{equation}
with
\[
\red{E(v,w)}:=\frac12w^2+G(v) \quad \mbox{and} \quad
G(v):=-\int_{\rho_0}^{v}\widetilde{g}(s)\, ds.
\]
Let $(v,w)=(\varphi_1(x;v_0,w_0), \varphi_2(x;v_0,w_0))$ denote 
a solution of \eqref{ham} with $(v(0),w(0))=(v_0,w_0)$.
We deal with \eqref{ham} in 
$[\rho_{-1}, \frac{\gamma}{\beta}]\times \R$.
It is clear that $(\rho_{i},0), (\frac{\gamma}{\beta},0) \in [\rho_{-1}, \frac{\gamma}{\beta}]\times \R$ 
for each $i\in\{-1,0\}$ by \eqref{C}.
In view of the condition \eqref{C}, it follows that 
$\widetilde{g}(v)>0$ for $v\in (\rho_{-1}, \rho_0)$ and 
\red{$\widetilde{g}(v)<0$ for $v\in (\rho_0, \frac{\gamma}{\beta})$.}
Thus we make sure that $G$ takes its minimum at $\rho_0$
and $G(\rho_0)=0$
by the definition of $G$.
Letting $I:=(0,\min\{G(\rho_{-1}), G(\frac{\gamma}{\beta})\})$, 
we see that $E_c:=\{ (v,w)\in \R^2\,|\ E(v,w)=c\}$
forms a closed curve for each $c\in I$.
Therefore, we observe that if $G(v_0)\in I$, then all curves such that
$(v(0),w(0))=(v_0,0)$ are periodic 
\red{(let the period $L > 0$)}
and surround $(\rho_0,0)$. 
\red{
Since the closed curve of $E_c$ 
is symmetric with respect to the $v$-axis, 
there is a solution 
$(\varphi_1(x;v_0,0), \varphi_2(x;v_0,0))$ 
of \eqref{ham} 
such that 
$\varphi_1(\cdot;v_0,0)$ 
is strictly increasing in $[0, \frac{L}{2}]$, 
and such that  
\begin{align} \label{sennsai}
\varphi_1(x;v_0,0) = \varphi_1(L-x;v_0,0)
\quad \forall \, x \in \Big[0, \frac{L}{2}\Big].
\end{align}
}%

\medskip
\noindent
{\bf \boldmath Step 2. Construction of flat-hump-shaped stationary solutions.}
We consider the one-dimensional boundary-value problem 
\[
\begin{cases}
v''=\widetilde{g}(v) \quad \mbox{in} \ (0,\red{\ell}), \\
v'(0)=v'(\red{\ell})=0.
\end{cases}
\]
We first claim that 
\begin{equation}\label{douti}
G(\rho_{-1})>G(v_\lambda).
\end{equation}
Indeed, recalling the definition of $G$ with \eqref{tuikanomono}, 
we see from the condition \eqref{jyoukenn} that 
\begin{align*}
  G(\rho_{-1}) - G(v_\lambda) 
  &=\int_{\rho_{-1}}^{v_\lambda}(-\gamma \overline{f_\lambda}(s)+\beta s)\, ds \\
  &=-\gamma\int_{\rho_{-1}}^{v_\lambda}\overline{f_\lambda}(s)\, ds
  +\frac{\beta}{2}(v_\lambda^2-\rho_{-1}^2)
   >0.
\end{align*}
Therefore, we obtain \eqref{douti}.
We next observe that $G(v_\lambda)<G(\frac{\gamma}{\beta})$ 
since $v_\lambda\in(\rho_0, \frac{\gamma}{\beta})$ and $G$
is increasing in $(\rho_0, \frac{\gamma}{\beta})$.
This together with \eqref{douti} implies that $G(v_\lambda)\in I$ 
and there is $v_0\in(\rho_{-1},\rho_0)$ fulfilling 
$G(v_0)\in I$ and $G(v_0)>G(v_\lambda)$.
Applying Step 1, we deduce that there is 
a solution $(\varphi_1(x;v_0,0), \varphi_2(x;v_0,0))$ of \eqref{ham} 
\red{with period $\red{\ell}$} and there exists $x_1\in [0,\frac{\red{\ell}}{2}]$ such that 
$\varphi_1(x_1;v_0,0)=v_\lambda$. 
Thus we can define
\begin{equation}\label{Constract}
\begin{cases}
v(x):=\varphi_1(x;v_0,0) \quad &\mbox{for} \ x\in [0,\red{\ell}], \\
u(x):=f_\lambda(v(x)) &\mbox{for} \ x\in [0,x_1)\cup(\red{\ell}-x_1,\red{\ell}], \\
u(x):=1 &\mbox{for} \ x\in [x_1,\red{\ell}-x_1].
\end{cases}
\end{equation}

\medskip
\noindent
{\bf \boldmath Step 3. Conclusion.}
We check that the couple $(u,v)$ is a stationary solution of \eqref{Sys:Main} 
in the sense of Definition \ref{Df:SS}. Since 
$\varphi_2(x;v_0,0)\in C^1([0,\red{\ell}])$, it follows from \eqref{ham} 
that $v\in C^2([0,\red{\ell}])$. 
\red{
Also 
by the definition of $u$ (see \eqref{Constract}) 
we have 
$u \in H^1(0,\ell)$,
which guarantees that 
$\mathcal{D}_1(u) \in H^1(\red{0,\ell})$.} 
Thus we obtain (i) in Definition \ref{Df:SS}. 
As for (ii) in Definition \ref{Df:SS}, 
we can confirm \eqref{Con:uvbounds}, \eqref{Eq:2nd} and \eqref{Eq:B} 
by \eqref{Constract}, and so it remains to verify \eqref{Eq:1st}.
When $x\in [x_1,\red{\ell}-x_1]$, it is clear that 
$D_1(u)D_2(v)u'-h_1(u)h_2(v)v'=0$ from \eqref{Con:Dh2} and \eqref{Constract}.
Let $x\in [0,x_1)\cup(\red{\ell}-x_1,\red{\ell}]$. Then again by \eqref{Constract} we have 
\[
\frac{D_1(u)}{h_1(u)}u'-\frac{h_2(v)}{D_2(v)}v'
=[j(u,v)]' 
=[j_1(u)-j_2(v)]'
=0.
\]
\red{
Thus we arrive at \eqref{Eq:1st}, 
and hence $(u,v)$ is a stationary solution 
of \eqref{Sys:Main}. 
Moreover, we see that 
$(u,v)$ is flat-hump-shaped. 
Indeed,
the function $u$ is strictly increasing 
in $[0, x_1)$, 
which is guaranteed 
since  
the function $f_\lambda$ 
is strictly increasing 
by Lemma \ref{L}
and 
the function $v$ is also  
strictly increasing in $[0, \frac{\ell}{2})$
due to the definition of $v$ 
together with 
the monotonicity of $\varphi_1$ 
deduced from Step 1, 
and 
strictly decreasing 
in $(\ell-x_1, \ell]$  
as well as 
\eqref{sennsai} and \eqref{Constract}
hold.}
Finally, we assume \eqref{j1-1'}. In order to prove that $u'\in L^2(0,\red{\ell})$
it is enough to check that $u'\in L^2(x_1-\ep,x_1)$ 
and $u'\in L^2(\red{\ell}-x_1-\ep,\red{\ell}-x_1)$ for some 
$\ep>0$ in view of \eqref{Constract}.
Due to \eqref{Constract}, it follows that 
\begin{align*}
\int_{x_1-\ep}^{x_1} |u'(x)|^2 \, dx
&=\int_{x_1-\ep}^{x_1}|(j_1^{-1})'(j_2(v(x))+\lambda)|^2 |j_2'(v(x))v'(x)|^2 \, dx \\
&=\int_{j_2(v(x_1-\ep))+\lambda}^{j_2(v(x_1))+\lambda}
|(j_1^{-1})'(\sigma)|^2 
\frac{|j_2'(j_2^{-1}(\sigma)-\lambda)|^2}{|j_2'(j_2^{-1}(\sigma))|}
\cdot v'(v^{-1}(j_2(\sigma)-\lambda))\, d\sigma.
\end{align*}
Setting $\ep >0$ such that $j_1(1)-\delta = j_2(v(x_1-\ep))+\lambda$, 
we obtain that
$u'\in L^2(x_1-\ep,x_1)$ by \eqref{j1-1'}.
Similarly, we can observe that \red{$u'\in L^2(\red{\ell}-x_1,\red{\ell}-x_1+\ep)$} for some 
$\ep>0$, which completes the proof.
\end{proof}
We now give sufficient conditions 
for \eqref{C} in the following two propositions.
\begin{prop}
Let $\widetilde{\lambda}
:=j_1(1)-\red{j_2\big(\frac{\gamma}{\beta}\big)}$.
Assume that 
\begin{equation}\label{bibunnC}
\lim_{s \nearrow v_{\widetilde{\lambda}}} 
f_{\widetilde{\lambda}}'(s)>\red{\frac{\beta}{\gamma}}.
\end{equation}
Then there exists $\ep_0>0$ such that the triplet 
$(\lambda, \gamma, \beta)$ fulfills \eqref{C} 
for all $\lambda\in (\widetilde{\lambda},\widetilde{\lambda}+\ep_0)$.
\end{prop}
\begin{proof}
We first note from the choice of $\widetilde{\lambda}$ 
and the definitions of $v_\lambda$ and $f_\lambda$ that 
$v_{\widetilde{\lambda}}=j_2^{-1}(j_1(1)-\widetilde{\lambda})
=\frac{\gamma}{\beta}$ and 
$\lim_{s\to -\infty}f_{\widetilde{\lambda}}(s)=0$.
By means of \eqref{bibunnC}, the graph of $r=f_{\widetilde{\lambda}}(s)$ 
has an intersection point 
$(w,f_{\widetilde{\lambda}}(w))$ 
fulfilling $w<v_{\widetilde{\lambda}}$ 
with the graph of  $r=\frac{\beta}{\gamma}s$.
Fix $\ep>0$ small enough.
Then we move the graph of $r=f_{\widetilde{\lambda}}(s)$ 
to the left by $\ep$.
Set 
\[
\widehat{\ep}(s):=j_2(s+\ep)-j_2(s)
\]
\red{for $s\in [0, \infty)$.
Let $\lambda\in (\widetilde{\lambda}, \widetilde{\lambda}+\ep_0)$ and
set 
\[
\ep_0:=\inf_{s \in [0, v_{\lambda}]}\widehat{\ep}(s).
\]
Then for all $s\in [0,\infty)$,}
we see that
\[
F_{\lambda}(s)=j_1^{-1}(j_2(s)+\lambda)>j_1^{-1}(j_2(s)+\widetilde{\lambda})
=F_{\widetilde{\lambda}}(s)
\]
since $j_1^{-1}$ is strictly increasing.
On the other hand, we have
\begin{align*}
F_{\lambda}(s)&<j_1^{-1}(j_2(s)+\widetilde{\lambda}+\ep_0)\\
&\le j_1^{-1}(j_2(s)+\widetilde{\lambda}+\widehat{\ep}(s)) \\
&=j_1^{-1}(j_2(s+\ep)+\widetilde{\lambda}) \\
&=F_{\widetilde{\lambda}}(s+\ep).
\end{align*}
Therefore, we arrive at the conclusion.
\end{proof}
\begin{prop} \label{tekiyou}
Assume that 
\begin{equation}\label{j2hatto}
\lim_{r \to 1} j_1'(r)<\frac{\gamma}{\beta}j_2'\left(\frac{\gamma}{\beta}
\right).
\end{equation}
Put
\begin{equation}\label{r0}
r_0:=\sup\left\{r\in(0,1)\,\bigg|\,j_1'(r)
=\frac{\gamma}{\beta}j_2'\left(\frac{\gamma}{\beta}r\right)\right\}.
\end{equation}
If
\begin{equation} \label{ramuda}
j_1(1)-j_2\left(\frac{\gamma}{\beta}\right)<\lambda
<j_1(r_0)-j_2\left(\frac{\gamma}{\beta}r_0\right),
\end{equation} 
then the triplet $(\lambda, \gamma, \beta)$ fulfills the condition \eqref{C}.
\end{prop}
\begin{proof}
We define
\[
\psi(r):=j_2\left(\frac{\gamma}{\beta}r\right)-j_1(r)
\]
for $r\in(0,1]$.
Since $j_1'(r)=\frac{D_1(r)}{h_1(r)}$, the condition \eqref{Con:Dh1} 
implies $\lim_{r \to 0} \psi'(r) = -\infty$. 
On the other hand, it follows from \red{\eqref{j2hatto}} 
that $\lim_{r \to 1} \psi'(r) >0$.
Therefore, there exists $\widehat{r}\in(0,1)$ such that $\psi'(\widehat{r})=0$, and we set 
$r_0:=\sup\{ \widehat{r}\in(0,1) \, | \, \psi'(\widehat{r})=0 \}$ as in \eqref{r0}.
It is clear that $\psi'(r)>0$ for $r\in (r_0,1)$.
Let $\lambda \in (-\psi(1),-\psi(r_0))$. 
Then, defining $\psi_\lambda$ as $\psi_\lambda(r) :=\psi(r)+\lambda$ 
for $r \in (0,1]$, we have $\psi_\lambda(1)>0>\psi_\lambda(r_0)$.
Since $\psi_\lambda'(r)=\psi'(r)$ for $r\in(r_0,1)$, 
there exists a uniquely determined $s_0\in (r_0,1)$ such that 
$\psi_\lambda(r)<0$ for $r\in[r_0, s_0)$,
$\psi_\lambda(s_0)=0$ and
$\psi_\lambda(r)>0$ for $r\in (s_0,1)$.
Also, we know that $\psi_\lambda(r) \to \infty$ as $r \to 0$ 
from Lemma \ref{L}.
Therefore, there exists $\widehat{s}\in (0,r_0)$ 
such that $\psi_\lambda(\widehat{s})=0$ and we set 
$s_{-1}:=\sup\{ \widehat{s}\in(0,r_0) \, | \, \psi_\lambda(\widehat{s})=0 \}$. 
Then $\psi_\lambda(r)<0$ for $r\in (s_{-1},s_0)$. 
\red{
Recalling 
$v_\lambda =j_2^{-1}(j_1(1)-\lambda)$, 
we have $v_\lambda < \frac{\gamma}{\beta}$  
due to \eqref{ramuda}. 
Moreover, 
we observe that 
$\rho_0 < v_{\lambda}$ 
since 
it follows from 
$\psi_\lambda(s_0) = 0$ 
and 
$0 < j_1(1) - j_1(s_0)$ 
that 
$\psi_\lambda(s_0) < j_1(1) - j_1(s_0)$.} 
Setting $\rho_i:=\frac{\gamma}{\beta}s_i$ for $i\in \{-1,0\}$, we have \eqref{C} since
$f_{\lambda}(r)
=j_1^{-1}\big(\psi_\lambda\big(\frac{\beta}{\gamma}r\big)
                   +j_1\big(\frac{\beta}{\gamma}r\big)\big)$. 
\red{
Indeed, 
if $r \in (\rho_0, v_{\lambda})$, 
we see that 
$f_\lambda(r)>\frac{\beta}{\gamma}r$, 
and 
if $r \in (v_{\lambda}, \frac{\gamma}{\beta})$, 
then we have 
$\overline{f_\lambda}(r) 
= 1>\frac{\beta}{\gamma}r$}.
\end{proof}
We now provide an example fulfilling the assumption of Theorem \ref{SSE}.
\begin{ex}
An example of $D$ and $h$ satisfying the assumption 
of Theorem \ref{SSE} is given by
$h_1(r)=rD_1(r)$ and $h_2(s)=e^sD_2(s)$.
First, we claim that 
there is $(\lambda,\gamma,\beta)\in \R \times (0,\infty)^2$
that fulfills \eqref{C} by Proposition \ref{tekiyou}.
By the definitions of $j_1$ and $j_2$, we have 
\begin{align}
\label{YY}
&j_1(r)=\int_{\frac12}^{r}\frac{D_1(\sigma)}{h_1(\sigma)}\, d\sigma
=\int_{\frac12}^{r}\frac{1}{\sigma} \, d\sigma
=\log r -\log \frac12
=\log2r, \\
\label{j1}
&j_1'(r)=\frac{1}{r},\\
\label{LIM}
&\lim_{r \to 1} j_1'(r)=\lim_{r\to1}\frac{1}{r}=1,\\
\label{j2}
&j_2(s)=\int_{0}^{s}
\red{\frac{h_2(\sigma)}{D_2(\sigma)}}\, d\sigma
=\int_{0}^{s} e^\sigma \, d\sigma 
=e^s-1
\end{align}
for all $r\in[0,1]$ and $s\in[0,\infty)$.
Noting that
$\frac{\gamma}{\beta}j_2'(\frac{\gamma}{\beta})
=\frac{\gamma}{\beta}e^{\frac{\gamma}{\beta}}$,
we let $\eta_0$ denote a unique solution in $\R$ 
of the equation $\eta e^\eta=1$. Then $0<\eta_0<1$.
By virtue of \eqref{r0}, \eqref{j1}, the fact $j_2'(s)=e^s$ 
and the definition of $\eta_0$, we have
\begin{equation}\label{matar0}
 r_0=\frac{\beta}{\gamma}\eta_0.
\end{equation}
Let $\beta$ and $\gamma$ be such that $\frac{\gamma}{\beta}>\eta_0$ and 
\red{$\log2-e^{\frac{\gamma}{\beta}}+1<\lambda
<\log2r_0-e^{\frac{\gamma}{\beta}r_0}+1$.}
Since the conditions \eqref{j2hatto} and \eqref{ramuda} 
in Proposition \ref{tekiyou} are satisfied
by \eqref{YY}, \eqref{LIM} and \eqref{j2}, 
we see that the triplet $(\lambda,\gamma,\beta)$ 
fulfills \eqref{C}.

Next, we verify \eqref{jyoukenn}.
The left-hand side of \eqref{jyoukenn} is rewritten as
\begin{align*}
\int_{\rho_{-1}}^{v_\lambda} f_\lambda(\sigma) \, d\sigma
&=\int_{\rho_{-1}}^{v_\lambda} j_1^{-1}(j_2(\sigma)+\lambda) \, d\sigma \\
&=\int_{j_2(\rho_{-1})+\lambda}^{j_2(v_\lambda)+\lambda} j_1^{-1}(\sigma) 
\frac{1}{j_2'(j_2^{-1}(\sigma-\lambda))}\, d\sigma \\
&=\int_{j_1^{-1}(j_2(\rho_{-1})+\lambda)}^{j_1^{-1}(j_2(v_\lambda)+\lambda)}
\sigma j_1'(\sigma) \frac{1}{j_2'(j_2^{-1}(j_1(\sigma)-\lambda))} \, d\sigma.
\end{align*}
Also, since $\rho_{-1} >0$, we observe that
\[
\frac{1}{j_2'(j_2^{-1}(j_1(\sigma)-\lambda))}
\le \frac{1}{j_2'(j_2^{-1}(j_2(\rho_{-1})))}
=\frac{1}{\rho_{-1}}
=\frac{1}{e^{\rho_{-1}}}<1
\]
for $\sigma\in \big[j_1^{-1}(j_2(\rho_{-1})+\lambda), 
\red{ j_1^{-1}}(j_2(v_\lambda)+\lambda)\big]$.
Thus the facts that $j_1^{-1}(j_2(\rho_{-1})+\lambda)=s_{-1}$ 
and that $j_1^{-1}(j_2(v_\lambda)+\lambda)=1$ yield 
\begin{equation}\label{0}
\int_{\rho_{-1}}^{v_\lambda} f_\lambda(\sigma) \, d\sigma
<\int_{s_{-1}}^{1} \sigma j_1'(\sigma)\, d\sigma
=\int_{s_{-1}}^{1} d\sigma
=1-s_{-1}=1-\frac{\beta}{\gamma}\rho_{-1}.
\end{equation}
Here, we further assume that 
$\frac{\gamma}{\beta}>2$. 
\red{We fix 
$\delta > 0$ sufficiently small 
such that 
$0<\delta<\frac{\gamma}
{\beta}-2$} 
and let $\lambda:=j_1(1)-j_2\big(\frac{\gamma}{\beta}-\delta\big)$.
Then we have $v_\lambda=j_2^{-1}(j_1(1)-\lambda)
=\frac{\gamma}{\beta}-\delta>2$.
In order to show \eqref{jyoukenn} it is enough to prove that 
\[
\nu:=
\frac{\beta}{2\gamma}(v_{\lambda}^2-\rho_{-1}^2)
-\int_{\rho_{-1}}^{v_\lambda} f_\lambda(\sigma) \, d\sigma >0.
\]
Due to \eqref{0}, we see that
\begin{align}
\nonumber
\nu&>\frac{\beta}{2\gamma}(v_{\lambda}^2-\rho_{-1}^2)-1
+\frac{\beta}{\gamma}\rho_{-1}\\
\nonumber
&\ge\frac{\beta}{\gamma}
\left[\frac{v_\lambda}{2}(v_\lambda-\rho_{-1})
                -\frac{\gamma}{\beta}+\rho_{-1}\right] \\
\label{01}
&=\frac{\beta}{\gamma}
\left[\frac{1}{2}\left(\frac{\gamma}{\beta}-\delta-2\right)
                      \left(\frac{\gamma}{\beta}-\delta-\rho_{-1}\right) -\delta \right].
\end{align}
We estimate $\rho_{-1}$.
From the fact that $\psi_\lambda(s_{-1})=0$ 
it follows that  
\[
j_2(\rho_{-1})=j_1(s_{-1})-\lambda \le j_1(r_0)-\lambda.
\]
By virtue of
\eqref{YY}, \eqref{j2}, \eqref{matar0}, the fact that $\eta_0 e^{\eta_0}=1$ 
and the definition of $\lambda$, we have 
\[
e^{\rho_{-1}}-1 \le\log2r_0-\lambda 
=\log2+\log\frac{\beta}{\gamma}+\log \eta_0-\lambda 
=e^{\frac{\beta}{\gamma}-\delta}-\log\frac{\gamma}{\beta}-\eta_0-1.
\]
Thus we obtain
$
\rho_{-1}\le\log\big(e^{\frac{\beta}{\gamma}-\delta}
                             -\log\frac{\gamma}{\beta}-\eta_0\big),
$
which together with \eqref{01} implies 
\[
\nu\ge 
\frac{\beta}{\gamma}
\left[\frac{1}{2}\left(\frac{\gamma}{\beta}-\delta-2\right)
                            \left(\frac{\gamma}{\beta}-\delta 
                            -\log\left(e^{\frac{\beta}{\gamma}-\delta}
                            -\log\frac{\gamma}{\beta}-\eta_0\right)\right)-\delta \right]>0.
\]
\red{Therefore,} \eqref{jyoukenn} is fulfilled.
\end{ex}
\smallskip
\section*{Acknowledgments}
\red{
The authors deeply appreciate the 
anonymous reviewers 
for numerous suggestions
and constructive comments, 
which have led to significant improvements.}


\end{document}